\documentclass[12pt]{amsart}
\usepackage{amsfonts,latexsym,amsthm,amssymb,amsmath,amscd,euscript,bibentry}
\usepackage[margin = 3cm]{geometry}
\usepackage[shortlabels]{enumitem}
\usepackage[utf8]{inputenc}
\usepackage{hyperref}
\usepackage{arydshln}
\usepackage{dirtytalk}
\usepackage{float}
\usepackage{multirow}
\usepackage[normalem]{ulem}
\usepackage{amsmath, amssymb, amsfonts, amsthm, amscd}
\usepackage{manfnt}
\usepackage{fancyhdr}
\usepackage{caption} 
\usepackage{mathtools}
\usepackage{blkarray}
\usepackage{hyperref}
\usepackage{xcolor}
\usepackage{tikz}
\usepackage{CJKutf8}
\usepackage{graphicx}
\graphicspath{ {./images/} }
\usepackage{float}
\usepackage{fullpage}
\usepackage{url}
\usepackage[all]{xy}
   \SelectTips{cm}{10}
\usepackage{setspace}
\usepackage{mathrsfs}
\usepackage{txfonts}
\usepackage{comment}
\hypersetup{
    colorlinks=true,
    linkcolor=red,
    filecolor=red,      
    urlcolor=red,
}
 \renewcommand{\bar}{\overline}
\renewcommand{\bar}{\overline}
\newtheorem{theorem}{Theorem}[section]
\newtheorem{proposition}[theorem]{Proposition}
\newtheorem{Lemma}[theorem]{Lemma}
\newtheorem{corollary}[theorem]{Corollary}

\newtheorem{Conjecture}[theorem]{Conjecture}

\theoremstyle{definition}
\newtheorem{definition}[theorem]{Definition}

\theoremstyle{remark}
\newtheorem{remark}{Remark}

\usepackage{fancyhdr}
\usepackage{caption} 
\usepackage{thmtools}
\usepackage[framemethod=TikZ]{mdframed}
	\mdfdefinestyle{mdrecbox}
		{
			linewidth=0.5pt,
			skipabove=12pt,
			frametitleaboveskip=5pt,
			frametitlebelowskip=0pt,
			skipbelow=2pt,
			frametitlefont=\bfseries,
			innertopmargin=4pt,
			innerbottommargin=8pt,
			nobreak=true,
		}
	\declaretheoremstyle
		[
			headfont=\bfseries,
			mdframed={style=mdrecbox},
			headpunct={\\[3pt]},
			postheadspace={0pt},
		]
		{thmrecbox}
	\declaretheorem[style=thmrecbox,name=Example, numberlike=theorem]{examplebox}

\catcode`,\active

\catcode`\,12

\newcommand*\HYPERskip{&}
\catcode`,\active
\newcommand*\pfq{
\begingroup
\catcode`\,\active
\def ,{\HYPERskip}
\doHyper
}
\catcode`\,12
\def\doHyper#1#2#3#4#5{
\, _{#1}F_{#2}\left[\begin{matrix}#3 \smallskip \\  #4\end{matrix} \; ; \; #5\right]
\endgroup
}

\catcode`,\active

\catcode`\,12

\newcommand{\C}{\mathbb{C}}

\newcommand{\Q}{\mathbb{Q}}
\newcommand{\Z}{\mathbb{Z}}

\newcommand{\g}{\mathfrak{g}}

\newcommand{\p}{\mathfrak{p}}

\newcommand{\tr}[0]{\operatorname{tr}}

\newcommand{\Gal}{\operatorname{Gal}}

\newcommand{\ol}[1]{\overline{#1}}

\makeatletter
\newcommand*\bigcdot{\mathpalette\bigcdot@{.7}}
\newcommand*\bigcdot@[2]{\mathbin{\vcenter{\hbox{\scalebox{#2}{$\m@th#1\bullet$}}}}}
\makeatother

\newcommand{\0}{\circ}

\DeclareMathOperator{\Frob}{Frob}

\NewCommandCopy{\originalsqrt}{\sqrt}
\RenewDocumentCommand{\sqrt}{om}{%
  \IfNoValueTF{#1}{\originalsqrt{#2}}{%
    \mspace{-3mu}\originalsqrt[#1]{#2}%
  }%
}

\usetikzlibrary{cd}

\definecolor{purple}{rgb}{0.59, 0.44, 0.84}

\begin{document}
\title{The Arithmetic of Reducible Rank 2 Hypergeometric Motives}
\author{Esme Rosen }
\date{}

\begin{abstract}
We study the arithmetic of the realizations for hypergeometric motives attached to ${}_3F_2(1)$ hypergeometric series defined over number fields, focusing on the case where the \'etale realization is reducible and irregular. We then use motivic techniques to prove new evaluation formulas for certain ${}_3F_2(1)$ series, understood as a period of the hypergeometric motive. In addition, we provide examples which are closely related to the exact values of $L$-functions for modular forms in special cases.
\end{abstract}
\maketitle

\section{Introduction}

There has been a broad interest since the time of Gauss in studying evaluations of classical hypergeometric functions in terms of special values of the gamma function, $\Gamma(z)$. The classical ${}_2F_1(z)$ hypergeometric series is known to be the solution to a 2nd order differential equation with regular singularities at $0$, $1$, and $\infty$. The singularity at 1 is a \textit{quasi-reflection}, meaning the hypergeometric series satisfies extra symmetries. Gauss showed that if the ${}_2F_1(1)$ series converges, its value is a quotient of gamma values, namely \begin{equation}\label{gausseval}
    \pfq{2}{1}{a&b}{1&c}{1}=\frac{\Gamma(c-a-b)\Gamma(c)}{\Gamma(c-b)\Gamma(c-a)}. 
\end{equation}From a modern perspective, Gauss evaluation implies that the \textit{hypergeometric motive} \cite{patrikistaylot} attached to the series degenerates to a rank 1 Fermat motive \cite{otsubo-regular-fermat}, i.e. a motive arising from a Fermat Curve, $X^M+Y^M=Z^M$, which, in particular, admits complex multiplication (CM). See Remark \ref{sysreal} for further discussion about what precisely we mean by a motive. The key geometric background is that the underlying systems of realizations each have rank 1, and so there is exactly one period associated to it up to an algebraic multiple. 

 The generalized hypergeometric functions are still solutions of a rigid differential equation with three singularities, with a quasi-reflection at 1. Therefore, it is natural to study generalized hypergeometric series evaluated at 1, and indeed this has been done since the time of Kummer; see \cite{aar,bailey,slater} for many classical results. As a general rule, evaluation formulas for  ${}_{n+1}F_n(1)$ series are expected to be fairly rare, as the rank of the motive gets larger with $n$. However, in the simplest case beyond $n=2$ of ${}_3F_2(1)$ motives, the rank is only 2. Thus, although there is no analogue of Gauss evaluation in general, such evaluation formulas appear to occur relatively more frequently in this setting. Examples of evaluation formulas at 1 were found by Dixon and Whipple for certain special families of hypergeometric data. These existing formulas for classical ${}_3F_2(1)$ are seemingly modeled after Gauss evaluation and similar results for ${}_2F_1(z)$ series. For instance, Dixon's formula states that  \begin{equation}\label{dixon}
   \pfq{3}{2}{1/2&1/2&r}{1&1&3/2-r}1=\frac{ 2\Gamma(5/4) \Gamma(3/4 - r) \Gamma(3/2 - r)}{\sqrt{\pi}\Gamma(3/4)
  \Gamma(1 - r) \Gamma(5/4 - r)}.
\end{equation} This formula only features one gamma quotient, despite the hypergeometric motive having rank 2. In this paper, we find a new type of evaluation formula for certain $_3F_2(1)$ series involving two (presumably) distinct gamma quotients by a geometric argument. Assume the Galois representation realized from the ${}_3F_2(1)$ motive is reducible and irregular, i.e. the Hodge--Tate weights are \textit{not} distinct; see Section \ref{regiireg} for further discussion. Then roughly stated, our main theoretical result states that there exist gamma quotients $G_1$ and $G_2$ and algebraic numbers $\alpha_1$ and $\alpha_2$ so that $${}_3F_2(HD,1)=\alpha_1\cdot G_1+\alpha_2\cdot G_2.$$ A precise statement is given in Theorem \ref{thm:decomp}, and requires some technical assumptions. So, we provide one infinite family of examples below, stated using the beta function, $B(a,b)=\frac{\Gamma(a)\Gamma(b)}{\Gamma(a+b)}$.

\begin{theorem}\label{mainconj}
    For $m>4$ an integer, there are algebraic numbers $\alpha_1$ and $\alpha_2$ depending on $m$ so that $$\pi B(1/2-2/m,1/2-1/m)\pfq{3}{2}{\frac{1}{2}&\frac{1}{2}&\frac{1}{m}}{1&1&\frac{1}{2}-\frac{1}{m}}{1}=\alpha_1\cdot B\left(1/4,1/m\right)^2+\alpha_2\cdot B\left(3/4,1/m\right)^2,$$ where $B(a,b)$ is the beta function.
\end{theorem}
Although the two beta values on the right hand side of the equation look similar outwardly, in reality they (likely) have different transcendental parts, namely their quotient is transcendental for all but finitely many choices of $m$. In Theorem \ref{lvalirr}, we show that when $m=8,12,24$, the transcendental parts are Chowla--Selberg periods with distinct fundamental discriminants, and so are known to be algebraically independent. Thus, this is generally not a single term evaluation as in \eqref{gausseval} or \eqref{dixon} written using two linearly equivalent beta values. See Remark \ref{transcendental} for further discussion. The original inspiration for Theorem \ref{thm:decomp} was the formulas in the finite field setting \cite{greene,mccart}, which have have a similar shape to Theorem \ref{mainconj}, see Lemma \ref{deco}. Thus, in a sense, we are motivated in reverse from the usual approach of e.g. Greene \cite{greene}, who diligently proved finite field versions of many known complex analytic formulae.

In contrast to Gauss evaluation, which classically is proved by analytic methods, the formula in Theorem \ref{mainconj} has a geometric proof and does not involve explicit algebraic numbers. In the second part of this paper, we provide several interesting examples where the formulas can be made explicit, and are related to the $L$-values of Hecke characters, or equivalently $GL_2$ automorphic forms with CM. This is related to a conjecture of Deligne \cite{deligne-l-values}, who predicts that the transcendental part for the critical $L$-values of a motivic $L$-function should be computed as the determinant of the period matrix for the motive. In the setting of hypergeometric motives, the periods are multiples of classical hypergeometric series, and so if the Galois representation is automorphic, it should be possible to write exact values of the associated automorphic $L$-function in terms of hypergeometric series. This heuristic was first introduced by Zagier \cite{zagierarith}, though certain specific examples were known well before. Most such results are proved using analytic techniques, see e.g. \cite{aroramondaltu,lilongtu,rogerswanzucker,zagierarith}. On the surface, the method used in this paper, called the Explicit Hypergeometric Modularity Method (EHMM) \cite{EHMMII}, also uses an analytic approach. However, following the framework of \cite{EHMMIII,rosen2,proc}, the EHMM is deeply connected to the motivic perspective. As part of this approach, in this paper we prove isomorphisms between the algebraic de Rham cohomology groups and vector spaces of modular forms. Our main theorem involving $L$-values is Theorem \ref{lvalirr} below, relating the evaluation formulas for the ${}_3F_2(1)$ hypergeometric series above to the $L$-values for modular forms.

 Before stating the theorem, to be more precise, we define the hypergeometric datum $$HD(r,s):=\{\{1/2,1/2,r\},\{1,1,s\}\},$$ which we verify has a Betti--de Rham structure $\mathfrak{H}(HD,1)_\Q$ over $\Q$ of rank $2\varphi(M)$, where $M$ is the least common denominator of $r$ and $s$.\footnote{This fact is well-known to experts, but to our knowledge is described only as a vector space over $\C$ in the literature, see e.g. \cite{fedorov}, and so we sketch the details in this paper.} Refer to Definition \ref{hdgdr} for the formal definition of $\mathfrak{H}(HD,1)_\Q$. Assume $\mathfrak{H}^{p,q}(HD,1)_\Q$ denotes the differentials of type $(p,q)$. If $f$ is a newform of level $N$, let $\mathcal{V}_f$ denote the vector space generated $f$ and its twists by Dirichlet characters $\chi$ of conductor dividing $M$. Finally, given a fundamental discriminant $-D$, let $\Omega_{-D}$ denote the associated Chowla--Selberg period, \begin{equation}\label{chowl}
      \Omega_{-D}:=\sqrt{\pi}\left(\prod_{i=1}^{D-1}\Gamma(i/D)^{\chi_D(i)}\right)^{1/2h'(-D)},
\end{equation} 
where  $h'(-D)$ is $1/2,1/3$ for $D=-4,-3$  and is the class number $h(-D)$ otherwise (see, for example, \cite{lilongtu,zagier123}).

 \begin{theorem}\label{lvalirr}
 The following statements are true.
 \begin{enumerate}
     \item [\normalfont \textbf{I.}] There exists a set of $(r,s)$, $\mathbb{S}_2^{irr}$, with 32 elements, so that for each $(r,s)\in \mathbb S_2^{irr}$, there exist CM newforms depending on $r$ and $s$ of weight 3 denoted by $f_1$ and $f_2$, of discriminant $-D_1$ and $-D_2$ respectively, with $$\mathfrak{H}^{2,0}(HD(r,s),1)_\Q\cong \mathcal{V}_{f_1}\oplus \mathcal{V}_{f_2}.$$
     \item [\normalfont \textbf{II.}] Let $\phi$ denote an element of $\widehat{(\Z/M\Z)^\times}$, for $M$ the least common denominator of $r,s$, and 1/2. Suppose that \begin{equation}\label{gammai}L(f_1\otimes \phi,1)=\gamma_1(\phi)\cdot \pi^{-1}\Omega^2_{-D_1}\quad\quad L(f_2\otimes \phi,1)=\gamma_2(\phi)\cdot \pi^{-1}\Omega^2_{-D_2}\end{equation} for $\gamma_i(\phi)$ algebraic numbers. Let $\{(r_j,s_j)\}_{i=1}^{\varphi(M)}$ denote the subset of $\mathbb S_2^{irr}$ of pairs equivalent to $(c\cdot r,c\cdot s)\mod \Z$ for some $c\in (\Z/M\Z)^\times$. Then there exist constants $\alpha_{1,c}$ and $\alpha_{2,c}$ so that \begin{equation}\label{alphaic}\pi B(s_j-r_j,s_j)\pfq{3}{2}{1/2&1/2&r_j}{1&1&s_j}{1}=\alpha_{1,j}\cdot \Omega_{{-D_1}}^2+\alpha_{2,j}\cdot \Omega^2_{-D_2}.\end{equation} Moreover, the algebraic numbers $\alpha_{i,j}$ are explicit for all $i$ and $j$ if and only if $\gamma_{i}(\phi)$ are explicit for all $i$ each $i\in\{1,2\}$ and $\phi$ in $\widehat{(\Z/M\Z)^\times}$.

 \end{enumerate}
\end{theorem}
 
We provide several examples where the algebraic numbers in the evaluation formulas, and thus the $L$-values, can be made explicit. For instance, consider the datum $HD(1/8,3/8)$. Note $\mathfrak{H}(HD(1/8,3/8),1)_\Q$ is a rank 8 Hodge structure over $\Q$ with Hodge numbers $(4,0,4)$. Via Theorem \ref{lvalirr}, we have $$\mathfrak{H}^{2,0}(HD({1/8,3/8}),1)\cong \mathcal{V}_{f_1}\oplus \mathcal{V}_{f_2}.$$ A calculation shows that the newforms $f_1$ and $f_2$ have LMFDB labels 128.3.d.a and 128.3.d.b, with CM discriminants $-4$ and $-8$. To compute the $L$-value for these forms, we need evaluation formulas for each $(c/8,3c/8)\in \mathbb S_2^{irr}$, $c\in (\Z/8\Z)^\times=\{1,3,5,7\}$. When $c=1$, there are two such pairs: $(1/8,3/8)$ and $(1/8,11/8)$; when $c=3,5$ there is one pair respectively, $(3/8,9/8)$ and $(5/8,7/8)$; and for $c=7$, there are no pairs lying in $\mathbb S_2^{irr}$ satisfying the required congruence. We then observe that Dixon's formula \eqref{dixon} applies to all pairs listed above other than $(1/8,3/8)$; for example, one checks that $$\pi B(1/4,3/8)\cdot \pfq{3}{2}{1/2&1/2&1/8}{1&1&11/8}{1}=\frac{65\cdot 2^{1/4}}{3}\sin(\pi/8)\cdot \Omega_{-4}^2.$$ An argument using twists of $L$-values allows us to bootstrap these evaluations to produce the following: \begin{equation}\label{prelev}
  \frac{\sqrt{2}}{16}\cdot  \pi B(1/4,3/8)\cdot \pfq{3}{2}{1/2&1/2&1/8}{1&1&3/8}{1}=\sqrt{1+\sqrt{2}}\cdot \Omega_{-8}^2+\sqrt{-1+\sqrt{2}}\cdot \Omega_{-4}^2. 
\end{equation}  Theorem \ref{lvalirr} tells us that we should now be able to compute the $L$-values for $f_1$ and $f_2$, and indeed, we find that $$L(f_{128.3.d.a},1)=\frac{1}{4} \sqrt{-\frac{1}{2} + \frac{i}{2}}\pi^{-1}\Omega_{-4}^2 \quad\quad L(f_{128.3.d.b},1)=\frac{2\sqrt{2(1+\sqrt{2})}}{18}\cdot \pi^{-1}\Omega_{-8}^2.$$   
The details are proved in Section \ref{ex2}. In general, for the $(r,s)\in \mathbb S_2^{irr}$, we are able to make our $L$-values explicit in many cases. We also handle a few cases that are not irregular, but are CM, in Appendix \ref{imprim}.

\begin{theorem}\label{complex}
    Assume $f$ is one of the Hecke eigenforms listed in Table \ref{tab} excluding Columns 12 and 15. Then the constants $\gamma_i(\phi)$ in \eqref{gammai} can be made explicit for all $i$ and $\phi$ as in Theorem \ref{lvalirr}, and therefore the constants $\alpha_{i,c}$ in \eqref{alphaic} can be made explicit as well. In particular, we are able to find the exact $L$-values for all modular forms listed in in rows 1-4, 11, 13, and 14 of Table \ref{tab}.
\end{theorem}

 It would be very interesting to relate the results of Theorem \ref{mainconj} with exact $L$-values of either Hecke $L$-functions or Hilbert modular forms with complex multiplication. In addition, via Theorem \ref{lvalirr}, we prove an evaluation formula, in Column 12 of Table \ref{tab} for the datum $HD(1/24,5/24)$, which is not recovered from Theorem \ref{thm:decomp}. In short, the essential hypothesis of Theorem \ref{thm:decomp} is that a hypergeometric Galois representation is reducible, and this is not yet proven for $HD(1/24,5/24)$. However, we are able to prove the evaluation formula using the technique of Theorem \ref{lvalirr}. This raises several additional questions.

 \smallskip
 
 \noindent\textbf{Question 1}: Does the evaluation of ${}_3F_2(HD(1/24,5/24),1)$ fit into an infinite family like the well-poised data? 

 \smallskip

 A closely related question is whether it is possible to compute the constants $\alpha_1$ and $\alpha_2$ in Theorem \ref{mainconj} explicitly, perhaps using an analytic method. One hopes that these evaluations are related to $L$-values of other Hecke $L$-functions/Hilbert modular forms.

 \smallskip
 
 \noindent \textbf{Question 2}: In which cases can we make our evaluation formulas explicit?

 \smallskip

 An answer to Question 2 for $HD(1/24,5/24)$ would produce the exact $L$-values for the underlying modular forms in Theorem \ref{lvalirr}.

This paper is organized as follows. First, in the Preliminaries, we establish our notation and recall the necessary facts from the Appendix to \cite{rosen2}. Section \ref{proofmainconj} appeals to the theory of motives/systems of realizations to prove Theorem \ref{mainconj}. In the first part of Section \ref{prooflvalirr}, we review the Explicit Hypergeometric Modularity Method and  define what the set $\mathbb{S}_2^{irr}$ actually is. Then in the remainder of Section \ref{prooflvalirr}, we prove Theorem \ref{lvalirr} in two parts, and in Section \ref{ex}, we prove Theorem \ref{complex} by computing each $L$-value explicitly. In the appendices, we address two related issues outside the main scope of this paper. In Appendix \ref{imprim}, we compute the $L$-values for the degenerate cases arising in the tables of \cite{rosen2}, to complete the computation of the $L$-value for every modular form listed in that paper. In Appendix \ref{tables}, we write down some results about gamma quotients that simplify the exposition of Section \ref{ex}. The tables in this Appendix states results that are implicit in \cite{proc}, but which are difficult to see if you are unfamiliar with that paper. We hope this Appendix will be useful for other researchers working with explicit gamma quotients as well.

\subsection{Acknowledgments}
The author is pleased to thank Cameron Franc, Ling Long, Akio Nawagawa, David Roberts, and Fang-Ting Tu for helpful discussions about the content of this paper.

\section{Preliminaries}

\subsection{Special Functions}\label{spcfun}

The gamma function is defined as $$\Gamma(x)=\int_0^\infty t^{x-1}e^{-t}\, dt,$$ and the beta function is $$B(a,b)=\int_0^1t^{a-1}(1-t)^{b-1}\, dt.$$ The beta function is related to the gamma function via the formula \begin{equation}\label{betagamma}
    B(a,b)=\frac{\Gamma(a)\Gamma(b)}{\Gamma(a+b)}.
\end{equation} We will also frequently use the following well known properties of the gamma function: for any real number $x\not \in \Z$ and $x>0$, and positive integer $m$ \begin{itemize}
    \item  $\Gamma(x+1)=x\Gamma(x)$ (the functional equation) 
    \item $\displaystyle \Gamma(1-x)\Gamma(x)=\frac{\pi}{\sin(\pi x)}$ (the reflection formula)
    \item $\displaystyle \prod_{k=0}^{m-1}\Gamma(x+k/m)=(2\pi)^{\frac{m-1}{2}}m^{1/2-mx}\Gamma(mx)$ (the multiplication formula) 
\end{itemize}\label{gamma}
   We use the language multiplying by $m$ when using the multiplication formula. The properties above are proved in e.g. \cite{aar}. Finally, we introduce the notion of the rising factorial, or Pochhammer symbol \begin{equation}\label{poch}
       (a)_k=\frac{\Gamma(a+k)}{\Gamma(a)}=a(a+1)(a+2)\dots (a+k-1).
   \end{equation} 

 Define a pair of multisets $$HD=\{\{a_1,\dots a_n\},\{1, b_1\dots b_{m-1}\}\}$$ as a \textit{hypergeometric datum} and the hypergeometric function $${}_nF_m(HD,z)=\sum_{k=0}^\infty\frac{(a_1)_k(a_2)_k\dots(a_n)_k}{k!(b_1)_k(b_2)_k\dots (b_{m-1})_k}z^k.$$  This is the generalized hypergeometric series, as introduced by Gauss. We let $M$ denote the least common positive denominator of all $a_i$ and $b_i$. The values of $M$ depends on $HD$, but this should always be clear from the context.

There is also an integral definition of a hypergeometric series, \begin{equation}\label{hypint}
    \pfq{n+1}{n}{a_0\dots a_{n-1},r}{1\dots b_{n-1},s}{z}=\frac{1}{\mathcal{B}(HD)}\int_0^1t^{r-1}(1-t)^{s-r-1}\pfq{n+1}{n}{a_0\dots a_{n-1}}{1\dots b_{n-1}}{tz}dt
\end{equation} where \begin{equation}\label{normfactor}
    \mathcal{B}(HD)=B(r,s-r)\prod_{i=1}^{n-1}B(a_i,b_i-a_i). 
\end{equation}

The period normalized hypergeometric series is 
\begin{equation}\label{frs}
    P(HD,1):=\mathcal{B}(HD)\cdot F(HD(r,s),1).
\end{equation}
We see below that this provides the transcendental part of the periods of our hypergeometric de Rham structure.

   \begin{definition}\label{ar}
       A hypergeometric datum is referred to as \textit{arithmetic} if the values $a_i$ and $b_i$ are rational numbers and $n=m+1$. If $a_i-b_j\not\in\mathbb Z$ for all $i$ and $j$, we call the datum \textit{primitive}.
   \end{definition}

The imprimitive hypergeometric series reduce to hypergeometric series of shorter length. From a geometric point of view, the \textit{monodromy group} is reducible if and only if the datum is imprimitive. In the special case of length 2 data, the imprimitive monodromy groups are in fact abelian; refer to \cite{beukers}. 

\subsection{Hypergeometric Betti--de Rham structures}

We recall the formalism of Betti--de Rham structures, which is a variant of systems of realizations. See \cite{huber-realization} for more information about systems of realizations, and refer to the introduction of \cite{deligneraghuram} for an overview of Betti--de Rham structures. Assume $K$ is a number field with fixed algebraic closure $\bar{K}$ and embedding into $\C$. \begin{definition}
   A Betti--de Rham structure $M$ over $K$ is a system consisting of \begin{enumerate}
       \item A $K$-vector space $M_{dR}$ equipped with a finite decreasing filtration, the Hodge filtration,
       \item A $K$-vector space $M_B$ together with an involution induced from complex conjugation, and
       \item A comparison isomorphism $M_{dR}\otimes \C\cong M_B\otimes \C$ equivariant with regard to the Hodge filtration and complex conjugation.
   \end{enumerate} 
   The \textit{period matrix} of a Betti--de Rham structure is defined to be the matrix giving the isomorphism of $\C$-vector spaces in (3) with a fixed choice of basis. So long as a $K$-rational basis is chosen for $M_B$ and $M_{dR}$ initially, the period matrix is well-defined up to multiplication by an invertible matrix with algebraic entries.
\end{definition}

If $M$ is a motive over $K$, its Betti and de Rham realizations give rise to a Betti--de Rham structure over $K$. Therefore, if $K\supset\Q(\zeta_M)$, by Patrikis--Taylor \cite{patrikistaylot} there is a Betti--de Rham structure attached any primitive and arithmetic datum at fibers $t\in \mathbb A^1(K)\setminus\{0,1\}.$ This generalizes work of  Katz \cite[Thm 8.4.2]{katzrig}, which defined $M_{dR}\otimes_{\bar{\Q}} C$, for $C=\C$ or $\bar{\Q_{\ell}}$, geometrically.
At the degenerate fiber $t=1$, the monodromy of Katz's local systems is a quasi-unipotent pseudoreflection, and the above constructions do not apply \textit{a priori}. In Appendix II of \cite{rosen2}, we gave a heuristic for how this is done for $HD$ of length 3 by constructing a nonsingular projective surface $\bar{C}_{HD,t}$ for which $P(HD,t)$ is a period. We will avoid the formalism of motives and work in the categories of cohomology for the purposes of this paper; again, see Remark \ref{sysreal}. The construction is explained below. For the hypergeometric datum $HD=\{\{a_1,a_2,a_3\},\{1,b_1,b_2\}\}$, let $A_i=M(1-a_i)$ and $B_i=M(1+a_i-b_i).$ Then define the affine hypersurface \begin{equation}\label{eq:chdt}
    C_{HD,t}: Y^M=X_1^{A_2}X_3^{A_3}(1-X_1)^{B_1}(1-X_2)^{B_2}(1-tX_1X_2)^{A_1}.
\end{equation} There is an action of $\mu_M$, the group of $M$th roots of unity on $C_{HD,t}$, given by sending $Y$ to $\zeta_M^{-1}Y$. The surface $\bar{C}_{HD,t}$ is defined as a projective desingularization of $C_{HD}$ equivariant under the action of $\mu_M$; we know such a compactification exists by \cite{normalcrossing}, and the choice of resolution does not impact the construction described in \cite{rosen2}, refer to \cite{patrikistaylot} for further discussion. Let $PH^2_{\bullet}(\bar{C}_{HD,1},K)$ denote the primitive part of the 2nd cohomology, where $K$ is a number field and $\bullet$ is a stand-in for either Betti or de Rham cohomology. Assuming $\Q(\zeta_M)\supset K$, the action of the primitive $M$th roots of unity acts on both cohomologies, inducing a decomposition $$PH^2_{\bullet}(\bar{C}_{HD,1},K)\cong \bigoplus_{n=1}^N V_{n,t,\bullet}.$$ Note we fixed an isomorphism between roots of unity and elements of $\Z/M\Z$, thus fixing a generator. \begin{definition}\label{hdgdr}
    We define the hypergeometric Betti--de Rham structure, $\mathfrak{H}(HD,t)$, as the isotypical components $V_{1,t,dR}$ with the Hodge filtration induced from $PH^2_{dR}(\bar{C}_{HD,t},K)$ and the Betti cohomology space $V_{1,t,B}$, together with the natural isomorphism $V_{1,t,dR}\otimes_K \C\cong V_{1,t,B}\otimes_K \C$. Use $\mathfrak{H}(HD,1)\otimes \C$ to denote either of the isomorphic $\C$-vector spaces above.
\end{definition}  By an abuse of notation, we write $\mathfrak{H}^{p,q}(HD,1)$ for $(p,q)$th piece of the Hodge decomposition for $V_{1,t,dR}$.  We also define the hypergeometric Betti--de Rham structure over $\Q$ as the rank $2\varphi(M)$ vector space $$\mathfrak{H}(HD,1)_\Q:=\bigoplus_{n\in (\Z/M\Z)^\times} \mathfrak{H}(n\cdot HD,1),$$ where $n\cdot HD$ denotes multiplying each entry of $HD$ by $n$ and reducing modulo $\Z$. Since the subspace of $PH^2_{\bullet}(\bar{C}_{HD,1},K)$ is stabilized by $\Gal(\Q(\zeta_M)/\Q)$, $\mathfrak{H}(HD,1)_\Q$ can be defined with coefficients in any finite extension of $\Q$. As above, we will take $\mathfrak{H}^{p,q}(HD,1)$ for $(p,q)$th piece of the Hodge decomposition for the de Rham incarnation, viewed as a vector space over $\Q$. Assuming $M\mid 24$, in \cite{rosen2}, we see by an explicit calculation using differentials, that the holomorphic periods of $\mathfrak{H}(HD(r,s),1)$ are $P(HD(r,s),1)$ and $P(HD(r,s+1)$ if $\dim \mathfrak{H}^{2,0}(HD(r,s),1)=2$. Regardless, $P(HD,1)$ is always a holomorphic period of $\mathfrak{H}^{2,0}(HD(r,s),1)$, if any exist.

\section{Proof of Theorem \ref{mainconj}}\label{proofmainconj}

The main goal of this section is to prove Theorem \ref{thm:decomp}, which has Theorem \ref{mainconj} as a corollary. Before proving the theorem, we review some basic information about the Gross--Deligne conjecture. If the conjecture is true, the evaluation formulas we obtain are actually a precisely determined quotient of gamma values.

\subsection{Gross--Deligne Periods}

The Gross--Deligne conjecture was initially conceived by Gross \cite{grossror} with input from Deligne as a way to compute the periods of CM motives. Assume that $X$ is a smooth projective variety over $K$, and $F$ is a subfield of $K$ and also an abelian extension of $\Q$. A Betti--de Rham structure $H$ is said to have CM by $F$ if $F$ can be embedded into $\text{End}_{Hdg}(H)\otimes \Q$ and $\dim_\Q H=[F:\Q]$, where $\text{End}_{Hdg}(H)$ denotes the Hodge structure endomorphism group. Then $$H\otimes_{\Q}\C=\prod_{\sigma\in \text{Gal}(F/\Q)} H_\sigma.$$ Each subspace $H_\sigma$ is of rank 1, and so determines a unique period $P(H_\sigma)$ depending on $\sigma$, up to multiplication by an algebraic number. The Gross--Deligne conjecture says that $P(H_\sigma)$ has a precise form as a gamma quotient. Since $F$ is abelian, there exists an integer $d$ so that $F\subset \Q(\zeta_d)$, which we choose to be minimal. Let $\sim$ denote equality up to an algebraic number.

\begin{Conjecture}[The Gross--Deligne Conjecture]\label{grossdeligne-l-values}
   We have $$P(H_\sigma)\sim \prod_{a=1}^{d-1}\Gamma(1-1/d)^{\varepsilon(a/\sigma)},$$ where $\varepsilon(a/\sigma)$ is a (non-unique) function from $\Z/d\Z$ to $\Q$ determined by $F$ and $\sigma$. 
\end{Conjecture}
The precise definition of $\varepsilon(a/\sigma)$ is somewhat delicate in general. We refer the reader to the introduction of \cite{fresan} for an excellent exposition. In the special case of Hodge structure with CM by $F=\Q(\sqrt{-D})$, the Gross--Deligne period is called the Chowla--Selberg period is defined earlier \eqref{chowl}. In this case, $\varepsilon(a/\sigma)=\chi_D(a)h'(D)/2$, where $\chi_D$ is the quadratic character attached to $F$ and $h'(D)$ is a normalized class number. This indicates how the $\varepsilon$ function value encodes arithmetic information about $F$,

Restricted to out setting, we have $\dim_\Q \mathfrak{H}(HD(r,s),1)_\Q=2\varphi(M)$, but we also have that $\Q(\zeta_M)\subset \text{End}_{Hdg}(\mathfrak{H}(HD(r,s),1)_\Q)\otimes \Q$, via the construction of the spaces $\mathfrak{H}(HD(r,s),1)$. As a result, if the hypergeometric Betti--de Rham structure is CM, there are two possibilities: \begin{enumerate}
    \item There is a degree 2 extension of $\Q(\zeta_M)$, $K$, so that $K\subset \text{End}(H)\otimes \Q$; or
    \item The space $\mathfrak{H}(HD(r,s),1)_\Q$ can be decomposed into sub-Hodge structures, each having CM by a subfield of $\Q(\zeta_M)$.
\end{enumerate}
We provide examples of the second phenomenon. It would be interesting to find an example of the first case, especially if the field $K$ had non-abelian Galois group, though we currently have no evidence that such an example exists. We leave the exploration of this problem to future work.

\subsection{Evaluation Formulas}

  Consider a Hecke character $\chi$ on a field $F$. It is well-known that we can attach to $\chi$ a 1-dimensional $\ell$-adic representation of $\Gal(\bar{F}/F):=G_F$, see e.g. \cite{raghuram}, which we will refer to as $\chi_\ell$. Suppose there is a motive $M$ over $\Q$ that admits CM by $F$, i.e. the Betti--de Rham structure induced from the realization functors has CM. As in the case of Betti--de Rham structures, we can decompose the base change of $M$ to $F$, $M_F$, as $M_F\cong \prod M_\sigma$ via automorphisms $\sigma\in \text{Gal}(F/\Q)$. Suppose there is such a $\sigma$ so that the factor $M_\sigma$ has \'etale realization whose action of $G_F$ is isomorphic to $\chi_\ell$ for all $\ell$. Then we refer to $\chi$ as \textit{motivic}. Conjecturally, all algebraic Hecke characters are motivic \cite{raghuram}. Note the de Rham realization of $M_\sigma$ has a period defined as above as well.

\begin{remark}\label{sysreal}
    We are intentionally vague about the category of motives used in this paper, since throughout, when we refer to motives, it suffices for the theorem if their exists an system of realizations. We assume that all motives are pure throughout. Patrikis--Taylor \cite{patrikistaylot} showed hypergeometric motives exist in Andr\'e's category of motivated pure motives \cite{andre-constructing-motives}, so it is also generally adequate to assume that all motives lie in that category. However, there are some slight inconveniences; for instance, the Fermat motives in \cite{otsubo-regular-fermat} are defined in the category of Chow motives. Ultimately, for this paper the formalization of motives which we use is unimportant, as we are concerned with the transcendental invariants encoded by cohomology rather than the theory of cycles on varieties. See \cite{ayoub-motivic-sheaf} for further discussion of this dichotomy.
\end{remark}

 Katz \cite[Theorem 8.4.1]{katzrig} showed that the Galois representations $\rho_{HD,\ell,1}$ are motivic, i.e. there exists a variety $X$ so that $\rho_{HD,\ell,1}$ is induced from a factor of the $\ell$-adic cohomology $H^2_{et}(X_{\bar{\Q}},\overline{\Q_\ell})$. We denote this factor by $H^2_{et}(M_{HD},\overline{\Q_\ell})$, where $M_{HD}$ indicates this is the \'etale realization of the hypergeometric motive attached to $HD$.
To avoid technical discussions about hypergeometric motives at 1, we do need the following lemma to relate the Betti--de Rham structure $\mathfrak{H}(HD(r,s),1)$ described in Definition \ref{hdgdr} to $H^2_{et}(M_{HD},\bar{\Q_\ell})$.

\begin{Lemma}\label{comp}
    There is an isomorphism of $\C$-vector spaces $$\mathfrak{H}(HD,1)\otimes \C\cong H^2_{et}(M_{HD},\bar{\Q_\ell})\otimes_{\Z_\ell}\C.$$
\end{Lemma}
\begin{remark}
    The first part of the argument is essentially due to Katz, but we provide a sketch of the proof for the readers convenience.
\end{remark}
\begin{proof}
  Katz \cite{katz,katzrig,katz-dwork-family} has shown that there is an irreducible lisse sheaf $\mathcal{H}^{can}(HD,\bar{\Q_\ell})$ defined over $\mathbb A^1(\Q(\zeta_M))\setminus\{0,1\}$ whose fibers are of rank 3, and whose Galois representation is $\rho_{HD,\ell,t}$ up to a twist on the base. The complexification $\mathcal{H}^{can}(HD,\bar{\Q_\ell})\otimes_{\bar{\Q}}\C$ is naturally isomorphic the local system associated to the hypergeometric differential equation, which we write as $\mathcal{L}_{HD}$. It is classical that 1 is a quasi-reflection for $\mathcal{L}_{HD}$, and so the fiber at 1, $(\mathcal{L}_{HD})_1$, is associated to a 2-dimensional complex vector space; see e.g. \cite{beukers} for a reference. Similarly, Katz's work shows that the singularity of $\mathcal{H}^{can}(HD,\bar{\Q_\ell})$ at 1 is also a quasi-reflection, and so $(\mathcal{H}^{can}(HD,\bar{\Q_\ell}))_1$ has a rank 2 subspace, which is precisely $H^2_{et}(M_{HD},\bar{\Q_\ell})$ by definition. It is thus clear that $H^2_{et}(M_{HD},\bar{\Q_\ell})\otimes_{\bar{\Q}}\C\cong(\mathcal{L}_{HD})_1$ as $\C$-vector spaces. 

  Likewise, $\mathfrak{H}(HD,1)\otimes \C$ is the singular fiber for a vector bundle. Write $C_{HD}$ for the variety $C_{HD,t}$ where $t$ is treated as a variable. Define the projection $\pi:C_{HD}\to \mathbb A^1$. Following the \cite[Theorem 8.4.1]{katzrig}, let $\underline{\C}$ denote the constant sheaf on $C_{HD}$. We then define the local system on $\mathbb A^1(\C)\setminus\{0,1\}$, $\mathcal{K}_\chi:=(\text{Gr}_r^WR^n\pi_!\underline{\C})^\chi$, where $\text{Gr}_n^W$ is the weight grading, $R^n$ is the $n$th right derived functor, $\pi_!$ is the direct image of $\pi$ with compact support, and $\chi$ is character from $\mu_M$ to $\C^\times$, chosen so that the isotypical component agrees with the choice in Definition \ref{hdgdr}. By definition, the fibers are the $\chi$th component of the weight $n$ part for the de Rham cohomology over $\C$, which is precisely $\mathfrak{H}(HD,t)\otimes \C$ by \cite{grothendieck}. On the other hand, we have constructed $\mathfrak{H}(HD,t)\otimes \C$, viewed as the complexified de Rham cohomology, so that $\frac{dX}{Y}$ is a differential in our space, c.f. the construction in \cite[Appendix II]{rosen2}. By the Euler integral formula (see e.g. \cite[Equation (2.2.4)]{aar}), this implies in the unit disc around 0 that $P(HD,t)$ is a period of each fiber, and therefore is a local section of $\mathcal{K}^\chi$. Therefore, the local monodromy around 0 is hypergeometric. A similar argument can be made at $\infty$ using the standard analytic continuation formulas given, for instance applying the Euler integral formula to, in \cite[Theorem 2.3.2]{aar}. Then by rigidity \cite{katzrig} (see also \cite{beukers}), the sheaf is isomorphic to $\mathfrak{L}_{HD}$. Combined with the above argument, the claim follows.
\end{proof}

Given an $\ell$-adic Galois representation $\rho_\ell$, we will let $\rho_\ell^{ss}$ denote its semi-simplification. We will also view $\mathfrak{H}(HD,1)\otimes \C$ as the complexified de Rham cohomology, and will write $$H^2_{dR}(M_{HD},\C):=\mathfrak{H}(HD,1)\otimes \C$$ for clarity of presentation in the following proof.

\begin{theorem}\label{thm:decomp}
   Let $HD$ be an irregular length three datum so that the Hodge numbers of $\mathfrak{H}(HD,1)$ are $(2,0,0)$. Assume that for some $\ell$, there are motivic Hecke characters $\chi_{\ell,1}$ and $\chi_{\ell,2}$ so that $$\rho_{HD,\ell,1}^{ss}\cong \chi_{\ell,1}\oplus \chi_{\ell,2}$$ is an isomorphism of $G_{\Q(\zeta_M)}$ representations. Let $\mathcal{G}_1$ and $\mathcal{G}_2$ denote the periods for the motives attached to $\chi_{\ell,1}$ and $\chi_{\ell,2}$. Then there exist algebraic numbers $\gamma_1$ and $\gamma_2$ so that $$P(HD,1)=\gamma_1\mathcal{G}_1+\gamma_2\mathcal{G}_2.$$ Assuming the Gross--Deligne conjecture, $\mathcal{G}_1$ and $\mathcal{G}_2$ are some quotients of gamma values specified in Conjecture \ref{grossdeligne-l-values}.
\end{theorem}
Implicitly, $\chi_{\ell,i}$ are $G_{\Q(\zeta_M)}$ representations. However, note that their CM type may not be $\Q(\zeta_M)$, and so the Gross-Deligne periods may be associated to a subfield of $\Q(\zeta_M)$.
\begin{proof}
   The Galois representation $\rho_{HD,\ell,1}$ is naturally associated to its Galois module, $H^2_{et}(M_{HD},\bar{\Q_\ell})$. By assumption $\chi_{\ell,1}$ and $\chi_{\ell,2}$ have cohomology groups, $H^2_{et}(M_i,\overline{\Q_\ell})$, which fit into a system of realizations. By default, the Galois representations attached to $H^2_{et}(M_i,\bar{\Q_\ell})$ are semi-simple, as they are 1-dimensional. The isomorphism of Galois representations up to semi-simplification indicates that $H^2_{et}(M_{HD},\bar{\Q_\ell})$ is reducible, i.e. there is a non-trivial subrepresentation $\mathcal{W}$ of $H^2_{et}(M_{HD},\bar{\Q_\ell})$, and as $H^2_{et}(M_{HD},\bar{\Q_\ell})$ is degree 2, $\mathcal{W}$ must be degree 1. As a result, there is a filtration of the form $H^2_{et}(M_{HD},\bar{\Q_\ell})\supset \mathcal{W}\supset 0$, which by the Jordan--H\"older theorem is unique up to isomorphism. By definition, the semi-simplification of $\rho_{HD,\ell,1}$ is \begin{equation}\label{ss}
        \rho_{HD,\ell,1}^{ss}\cong H^2_{et}(M_{HD},\bar{\Q_\ell})/\mathcal{W}\oplus \mathcal{W}\cong H^2_{et}(M_1,\bar{\Q_\ell})\oplus H^2_{et}(M_2,\bar{\Q_\ell}),
    \end{equation} with the second isomorphism by hypothesis. However, when complexifying using Lemma \ref{comp} on the left and the motivic realizations on the right, the structure of Galois modules is lost, and so we simply obtain a vector space $H^2_{et}(M_{HD},\bar{\Q_\ell})\otimes_{\Z_\ell}\C\cong H^2_{dR}(M_{HD},\C)$, which is a 2-dimensional $\C$-vector space with 1-dimensional subspace $\mathcal{W}_\C:=\mathcal{W}\otimes_{\Z_{\ell}}\C$. Vector spaces are then ``semi-simple'', i.e. given a subspace like $\mathcal{W}_\C$, there is a complement $\mathcal{W}_\C'$ so that \begin{equation}\label{complexify}
        H^2_{dR}(M_{HD},\C)\cong \mathcal{W}_\C\oplus \mathcal{W}_\C'.
    \end{equation} On the other hand, complexifying \eqref{ss}, we get that \begin{equation}\label{sscomp}
        H^2_{dR}(M_{HD},\C)/\mathcal{W}_\C\oplus W_\C\cong H_{dR}^2(M_1,\C)\oplus H^2_{dR}(M_2,\C).
    \end{equation} Combining \eqref{sscomp} with \eqref{complexify},  we deduce that  \begin{equation}\label{iso}
        H^2_{dR}(M_{HD},\C)\cong H^2_{dR}(M_1,\C)\oplus H^2_{dR}(M_2,\C).
    \end{equation}
    
    By \cite{grothendieck}, $H^2_{dR}(M_{HD},\C)\cong H_{dR}(M_{HD},\bar{\Q})\otimes \C$. As in Kontesvich--Zagier \cite{kontsevichzagier}, one can define the algebraic period matrix as the integrals of differentials in $H_{dR}(M_{HD},\bar{\Q})$ over integral homology classes, and this period matrix is unique up to multiplication by invertible matrices with algebraic entries. The construction of the period matrix is analogous for the algebraic de Rham incarnation of $H^2_{dR}(M_i,\C)$, denoted by $H^2_{dR}(M_i,\bar{\Q})$.
 We now compute the period matrices in two different ways. Via the Euler integral formula, see e.g. \cite{fedorov,kellyvoight}, and because the Hodge numbers of $H^2_{dR}(M_{HD},\bar{\Q})$ are $(2,0,0)$, there is a choice of basis for $H^2_{dR}(M_{HD},\bar{\Q})$ so that the first entry of the period matrix is $P(HD,1)$ up to an algebraic multiple. We denote this matrix by $\text{Per}(M_{HD})$. The period matrix of $H^2_{dR}(M_i, \bar{\Q})$ is a number, whose transcendental part we denote by $\mathcal{G}_i$. As a result, the period matrix of $H^2_{dR}(M_1,\bar{\Q})\oplus H^2_{dR}(M_2,\bar{\Q})$ is $$\text{Per}(M_1\oplus M_2)=\begin{pmatrix}
        \mathcal{G}_1&0\\0&\mathcal{G}_2
    \end{pmatrix},$$ up to multiplication by an invertible diagonal matrix with algebraic entries. 
On other hand, the period matrix is also defined see e.g. \cite{deligneraghuram,otsubo-regular-fermat} as the matrix giving an isomorphism between the groups $H^2_{dR}(X,\bar{\Q})\otimes \C$ and $H^2_{B}(X,\bar{\Q})\otimes \C$. 
    Using \eqref{iso}, we obtain the chain of isomorphisms 
    $$\begin{tikzcd}
H^2_{dR}(M_{HD},\C) \arrow[r,leftrightarrow] \arrow[d,leftrightarrow]
& H^2_{dR}(M_1,\C)\oplus H^2_{dR}(M_2,\C) \arrow[d,leftrightarrow] \\
H^2_{dR}(M_{HD},\bar{\Q})\otimes \C  \arrow[d,leftrightarrow]
& (H^2_{dR}(M_1,\bar{\Q})\oplus H^2_{dR}(M_2,\bar{\Q}))\otimes \C\arrow[d,leftrightarrow]\\ H^2_{B}(M_{HD},{\Q})\otimes \C  
& (H^2_{B}(M_1,\Q)\oplus H^2_{B}(M_2,\Q))\otimes \C
\end{tikzcd}$$
It is known, see e.g. \cite{deligne-l-values}, that if the same $\bar{\Q}$-rational choice of basis is made, the period matrices constructed via integration or as described above coincide.
As we have chosen a $\bar{\Q}$-rational basis when constructing each of the matrices $\text{Per}(M_{HD})$ and $\text{Per}(M_1\oplus M_2)$ and the isomorphisms above are functorial, the two period matrices are the same up to a choice of $\bar{\Q}$-rational basis, i.e. there exist  $A,B\in GL_2(\bar{\Q})$ so that $$\text{Per}(M_{HD})=A\text{Per}(M_1\oplus M_2)B.$$ By construction, the first entry of $\text{Per}(M_{HD})$ is $P(HD,1)$ up to an algebraic multiple. Likewise, by the definition of $\text{Per}(M_1\oplus M_2)$, it is clear that after multiplying by matrices in $GL_2(\bar{\Q})$, the first entry will be a $\bar{\Q}$-linear combination of $\mathcal{G}_1$ and $\mathcal{G}_2$.  
\end{proof}

   The first paragraph can be simplified if one believes the Tate conjecture or the related semi-simplicity conjecture, see e.g. \cite{moonen}. In particular, these conjectures imply the representation $\rho_{HD,\ell,1}$ should be semi-simple already. Likewise, if we knew an isomorphism of motivated motives, $M_{HD}\cong M_1\oplus M_2$, \eqref{iso} would be immediate. However, in practice we find our assumptions to be most convenient for the applications we have in mind, specifically Theorem \ref{mainconj}.

\subsection{The Well-Poised Cases}
Finally, we will prove Theorem \ref{mainconj}. The key tool is known evaluation formulas for finite field hypergeometric functions, first provided by Greene \cite{greene} and McCarthy \cite{mccart}. To prove the precise version needed for this paper, we provide a rapid review of finite field hypergeometric functions.  Assume $r$ is a rational number and $M$ is its denominator. For $\zeta_M$ a primitive fixed $M$th root of unity, let $\mathfrak{p}$ be a prime ideal of $\Z[\zeta_M]$ and $k_\mathfrak{p}$ its residue field, which we assume is of size $q$. Then \cite{weil} there is a unique character on $k_\mathfrak{p}$
$$\iota_{\mathfrak{p},i/M}(x):=\left(\frac{x}{\mathfrak{p}}\right)_M^i\equiv x^{\frac{i(q-1)}{M}}\mod \mathfrak{p}.$$ Using this, define the \textit{Gauss sum} is $g_{\mathfrak{p}}(r):=\sum_{a\in k_\mathfrak{p}^\times}\iota_{\mathfrak{p},r}(a)\cdot \zeta_q^{\text{Tr(a)}},$ where $\text{Tr}(a)$ denotes the field trace from $k_\mathfrak{p}$ to its prime field. Gauss sums satisfy analogues to the reflection formula and multiplication formula at the start of Section \ref{spcfun}; see e.g. \cite{flrst} for precise statements. The definition of the \textit{Jacobi sum} can be taken as \begin{equation}\label{gauss}
    J_{\mathfrak{p}}(r,s)=\frac{g_{\mathfrak{p}}(r)g_{\mathfrak{p}}(s)}{g_{\mathfrak{p}}(r+s)},
\end{equation} 
for our purposes, in analogy with \eqref{betagamma} relating the beta and gamma function. Weil \cite{weil} showed that Jacobi sums are Gr\"ossencharaketere in the sense of Hecke, and thus correspond to algebraic Hecke characters. Likewise, the trace of the Katz representation is a character sum, known as a finite field hypergeometric function. We use the conventions of \cite{flrst} in defining them; see loc. cit. for other versions of finite field hypergeometric sums. Using Greene's \cite{greene}  finite field binomial coefficient for characters $A$ and $B$ on $\mathbb F_q$, given by $$\binom{A}{B}=-B(-1) \sum_{x=0}^{q-1}A(x)\bar{B}(1-x),$$ we define the finite field period function, \begin{equation}\label{period}
    \mathbb{P}(HD,\lambda,\mathfrak{p}):=\frac{(-1)^n}{q-1}\prod_{i=1}^n \iota_{\mathfrak{p},a_i}\iota_{\mathfrak{p},b_i}(-1)\sum_{\chi\in \widehat{k_\mathfrak{p}^\times}}\binom{\iota_{\mathfrak{p},a_0}\chi}{\chi}\binom{\iota_{\mathfrak{p},a_1}\chi}{\iota_{\mathfrak{p},b_1}\chi}...\binom{\iota_{\mathfrak{p},a_n}\chi}{\iota_{\mathfrak{p},b_{n}}\chi}.
\end{equation}
We normalize $\rho_{HD,\ell,t}$ so that $\text{tr}(\rho_{HD,z,\ell}(\Frob_\p))=(-1)^{n-1}\mathbb{P}(HD,z,\p)$; note this only requires taking a twist of the base, and so does not impact any of the arguments above. For Theorem \ref{mainconj}, we require the following lemma. This amounts to computing the Hecke character for $\chi_{\ell,1}$ and $\chi_{\ell,2}$ in Theorem \ref{thm:decomp} using Jacobi sums for these cases. This computation enables us to show that Hecke characters are motivic, and so we can apply Theorem \ref{thm:decomp}.

\begin{Lemma}\label{deco} We have 
  For $\p$ be a prime ideal of $\Q(\zeta_M)$ as above, let $q$ be the size of the residue field at $\p$, and let $c_r(\mathfrak{p})=\iota_{\mathfrak{p},r+1/4}(1/4)\cdot\iota_{\mathfrak{p},r+1/2}(-1)$. Then,   $$\mathbb{P}(HD(r,3/2-r),1,\mathfrak{p})=\begin{cases}
      c_r(\mathfrak{p}) \cdot J_\p(1/4,r)^2+c_r(\mathfrak{p})\cdot J_\p(3/4,r)^2 & 4\mid q-1\\0&4\nmid q-1
  \end{cases}$$  
\end{Lemma}

\begin{proof}
Set $$
\mathcal{G}_{r,\p}(i):=\frac{g_\p(1/2)g_\p(1/2-i/4)g_\p(r-i/4)g_\p(r)}{g_\p(1-i/4)g_\p(1)g_\p(r+1/2)g_\p(r+1/2-i/4)}.$$ It was proved in \cite{greene,mccart} that $$\mathbb{P}(HD(r,1/2-r),1,\mathfrak{p})=\begin{cases}  \displaystyle J_{\mathfrak{p}}(r,1/2-2r)\cdot [\mathcal{G}_{r,\p}(1)+\mathcal{G}_{r,\p}(2)] & \text{ if }4\mid q-1\\0& \text{ if }4\not\mid q-1. \end{cases}$$ Rewriting the Jacobi sums as Gauss sums, the first term becomes $$\frac{g_\p(1/2)g_\p(1/4)g_\p(3/4+r)g_\p(r)^2g_\p(3/2-2r)}{g_\p(3/4)g_\p(1/4+r)}.$$ We multiply top and bottom by $g_\p(1/4+r)g_\p(1/4)$ to get $$\iota_{\mathfrak{p},1/4+r}(1/4)\cdot \iota_{\mathfrak{p},1/4}(1/4)\cdot \frac{g_\p(1/2)g_\p(1/4)^2g_\p(1/2+2r)g_\p(r)^2g_\p(1/2-2r)}{g_\p(1/4)g_\p(3/4)g_\p(1/4+r)^2}$$
Simplifying using the reflection formula in denominator and the duplication formula $\iota_{\mathfrak{p},r+1/4}(1/4)\cdot g_\p(1/2+2r)g_\p(1/2)=g_\p(1/4+r)g_\p(3/4+r)$ in the numerator, we obtain
$$\iota_{\mathfrak{p},r} (-1)\cdot\iota_{\mathfrak{p},1/2+r} (-1)\cdot\iota_{\mathfrak{p},r+1/4}(1/4)\cdot \frac{g_\p(1/2)^2g_\p(1/4)^2g_\p(r)^2}{q\cdot g_\p(1/4+r)^2}.$$ Using that $\iota_{\mathfrak{p},1/2}(-1)\cdot q=g_\p(1/2)^2$, we cancel the $q$ and the three signs, producing the claimed character times a Jacobi sum. The process is identical for the second part of the sum.
\end{proof}

 We now prove Theorem \ref{mainconj}. Let us briefly explain the condition $m>4$ in that theorem. From Lemma \ref{galois}, when $r=1/m$ and $m>4$, the Hodge numbers of $\mathfrak{H}(HD(r,1/2-r),1)$ will be $(2,0,0)$, as is given in Theorem \ref{mainconj}. If $m<4$, the series ${}_3F_2(HD(r,1/2-r),1)$ doesn't even converge. Geometrically, this is explained as follows. The cases $m=2,4$ are imprimitive. This means that the Hodge numbers become $(1,0,0)$ and so there is exactly one holomorphic period, given by $B(1/4,1/m)$. It turns out that $B(3/4,1/m)$ is the non-holomorphic period, i.e. the period of the complex conjugate to our unique holomorphic differential. The case that $m=3$ has Hodge numbers $(1,0,1)$, and the situation is similar to above. 
Actually, the hypothesis is better stated using using Hodge numbers.

\begin{theorem}\label{mainconj2}
    Assume that $\mathfrak{H}(HD(r,1/2-r),1)$ is primitive and has Hodge numbers $(2,0,0)$. Then there exist constants $\gamma_1$ and $\gamma_2$ so that $$\pi B(1/2-2r,1/2-r)\pfq{3}{2}{\frac{1}{2}&\frac{1}{2}&r}{1&1&\frac{1}{2}-r}{1}=\alpha_1\cdot B\left(1/4,r\right)^2+\alpha_2\cdot B\left(3/4,r\right)^2.$$
\end{theorem}

\begin{remark}\label{transcendental}
   Let $\sim$ denote equality up to an element of $\bar{\Q}.$ In general, $$B(1/4,r)/B(3/4,r)\sim \Omega_{-4}/B(1/4-r,1/4+r).$$ When $r=1/6,5/6$, it is a straightforward but tedious exercise using the properties of the gamma function to show that $\Omega_{-4}/B(1/4-r,1/4+r)$ is algebraic; see Lemma \ref{betaval} for a related problem. Otherwise, it is apparent that this value will seldom be algebraic; see the appendix of \cite{deligne-l-values} by Koblitz--Ogus and Wolfart--W\"ustholz \cite{wolfart-wustholz} for a more precise description of linear relations among beta values.
\end{remark}

\begin{proof}[Proof of Theorem \ref{mainconj2}]
   Let $J_{\ell}(r,s)$ denote the $\ell$-adic character of $G_{\Q(\zeta_M)}$ associated to the Jacobi sum $J_\p(r,s)$. Set $$\chi_{\ell,1}=c_{\ell,r}\otimes J_{\ell}(1/4,r)^{\otimes 2}\quad\quad \chi_{\ell,2}=c_{\ell,r}\otimes J_\ell(3/4,r)^{\otimes 2},$$ where $c_{\ell,r}$ is the finite order character appearing in Lemma \ref{deco}. To apply Theorem \ref{thm:decomp}, we must check that 1) there is an isomorphism of Galois representations up to semi-simplification, 2) that $\chi_{\ell,i}$ are motivic, and 3) the periods of the motives for $\psi_{i,\ell}$ are the claimed beta values. From Lemma \ref{deco}, it is immediate that for all but finitely many prime ideals $\mathfrak{p}$ of $\Q(\zeta_M)$, $$\tr(\rho_{HD(r,1/2-r),\ell,1}(\Frob_\p))=\chi_{\ell,1}(\Frob_\p)+ \chi_{\ell,2}(\Frob_\p),$$ with $\chi_{\ell,i}$ as above. Hence, for 1), the Chebotarev density theorem implies that the semi-simplification, denoted by $ss$, satisfies $$(\rho_{HD(r,1/2-r),\ell,1})^{ss}\cong \psi_{1,\ell}\oplus\psi_{2,\ell}.$$  Moving to 2), since $r\notin\Z$ (by primitivity), we know $J_\ell(1/4,r)$ and $J_\ell(3/4,r)$ are motivic from \cite{otsubo-regular-fermat}, and arise from the \'etale realization of Chow motives. Let $\text{CH}(K,\Lambda)$ denote the tensor category of Chow motives over $K$ with coefficients in $\Lambda$. We denote these motives by $\mathcal{J}_{1/4,r},\mathcal{J}_{3/4,r}\in \text{CH}(\Q(\zeta_M),\Q(\zeta_M))$. There is also weight 0 motive attached to $c_\p$, which we call $M_c$, lying in $\text{CH}(\Q(\zeta_M),\Q(\zeta_M))$. To obtain a motive for $\chi_{\ell,i}$, we take $$\Psi_i:=M_c\otimes \mathcal{J}_{k/4,r}^{\otimes 2},$$ for $k=1$ when $i=1$ and $k=3$ otherwise. Because the realization functors are tensor functors (see \cite[Theorem 4.2.5.1]{andre-introduction}, the \'etale realization of $\Psi_i$ is precisely $\chi_{\ell,i}$. Therefore, the Hecke characters are motivic. Finally for 3), from \cite[Appendix]{grossror}, the period of $\mathcal{J}_{k/4,r}$ is precisely the beta value $B(k/4,r)$. It follows that transcendental part of the period for the tensor square is $B(k/4,r)^2$. The period for the weight zero part, $M_c$ is necessarily algebraic, and so can be absorbed into $\alpha_1$ and $\alpha_2$. The claim follows.
\end{proof}

 There is an analogue of Lemma \ref{deco} for the other well-poised hypergeometric data of length 3, which have the form $\{\{a,b,r\},\{1,1+a-b,1+a-r\}\}$ due to to Greene \cite{greene} and McCarthy \cite{mccart}.  
Clearly, it is possible to extend the proof above to these cases so long as the Hodge numbers are $(2,0,0)$, and the two holomorphic periods are known. One can check that when $a=b$, the Hodge numbers are $(2,0,0)$ if $a/2>r$. For instance, when $r=1/m$ and $a=1/b$, the Hodge numbers are $(2,0,0)$ if and only if $m>2b$. The principal difficulty is rewriting the Gauss sums as Jacobi sums, analogous to Lemma \ref{deco}. We leave this task to the reader. There are also other formulas similar to Lemma \ref{deco} for higher rank hypergeometric motives which show that the Galois representation $\rho_{HD,\ell,1}^{ss}$ is completely decomposable into 1-dimensional representations for certain choices of $HD$. See \cite[Theorem 4]{deinesetal2} for a few such examples. It may be possible to use the techniques of this paper to understand these cases, but an essential role is played by the Hodge numbers, and extra care would be required in understanding that aspect of the situation.

\section{Proof of Theorem \ref{lvalirr}}\label{prooflvalirr}
The goal of this section is to explain the relationship between the theory in the previous section to $L$-values of Hecke characters in some special cases. We rely on the techniques of the Explicit Hypergeometric Modularity Method (EHMM) for these examples. Before we  prove the theorem, we recall some preliminaries from previous papers on the subject and explain the concept of regularity in our context.
 Then, we will prove Theorem \ref{lvalirr} in two parts. 

\subsection{The \texorpdfstring{$\mathbb{K}_2$}{K2} Functions}

The $\mathbb K_2$ functions appearing in the EHMM are crucial for proving Theorem \ref{lvalirr}. Assume $\lambda$ is the modular lambda function. We define the functions, studied initially in \cite{EHMMI},
 \begin{align*}
    \mathbb{K}_2(r,s)(\tau)&=2^{-4r}\lambda^r(1-\lambda)^{s-r-1}\pfq{2}{1}{1/2&1/2}{1&1}{\lambda}q\frac{d}{dq}\log \lambda.
\end{align*}
In \cite{EHMMI}, they also show there is an eta quotient form for these forms: $$\mathbb{K}_2(r,s)(\tau)=\frac{\eta(\tau/2)^{16s-8r-12}\eta(2\tau)^{8r+8s-12}}{\eta(\tau)^{24s-30}}.$$ Using the eta quotient form, this is a weight 3 modular form, congruence if the exponents are integral, and holomorphic if $0<r<s<3/2$. The set of $(r,s)$ that are holomorphic and congruence is denoted $\mathbb S_2$, and $|\mathbb S_2|=199$. Consider the constant $$N(r)=\frac{48}{\text{gcd}(24r,24)}.$$ When $r$ is clear, we will often simply write $N$. From the eta quotient form, the authors of \cite{EHMMI} derived the level and character of $\mathbb{K}_2(r,s)(N\tau)$ which are $N(r)N(s-r)$ and the Dirichlet character induced from $\left(\frac{-2^{24s}}{\cdot}\right),$ respectively. The following lemma about Fourier expansions will be useful occasionally. \begin{Lemma}\label{coeffcong}[Lemma 3.2 of \cite{rosen2}]
    Assume $(r,s)\in \mathbb S_2$, and $r=m/e$ is in reduced terms. Then the $n$th Fourier coefficient of $\mathbb{K}_i(r,s)(N\tau)$ can only be nonzero if  $n\equiv m\mod e$. 
\end{Lemma} The $\mathbb{K}_2(r,s)(N\tau)$ forms are not Hecke eigenforms unless $(r,s)=(1/2,1)$. A method suggested in \cite{EHMMI} and fully worked out in \cite{rosen2}, see Theorem \ref{cons} below, implies that a linear combination of several such functions with varying $r$ and $s$ is a Hecke eigenform. Set   \begin{equation}\label{frs2}
    F(r,s):=\frac{2^{1-4r}\pi^{-1}}{N}\cdot P(HD(r,s),1).
\end{equation} By \cite{EHMMII}, we then have the formula \begin{equation}\label{lval}
    L(\mathbb{K}_2(r,s)(N\tau),1)=F(r,s).
\end{equation} Using the functional equation, one can also find the $L$-value at 2. We can also obtain the $L$-value of $f_{r,s}$ by this method as well by making an appropriate linear combination.

\begin{definition}\label{conj}
 Assume $(r,s)$ and $(r',s')$ are pairs of rational numbers in $\mathbb S_2$. If there exists an integer $n\in (\Z/M\Z)^\times$ so that $(r,s)\equiv (nr',ns')\mod \Z$, then we say the pair are \textit{conjugate}. We will frequently refer to the vector space of all conjugates corresponding to holomorphic $\mathbb{K}_2(r,s)$ by $\{(r_j,s_j)\}$, and let $\mathfrak{S}_{r,s}$ be the vector space over $\bar{\Q}$ generated by the corresponding $\mathbb{K}_2(r_j,s_j)$.
 \end{definition}  

We recall the following geometric interpretation of Definition \ref{conj} from \cite{rosen2}. 
\begin{Lemma}\label{geom}[\cite{rosen2}, Lemma 7.1]
    Assuming $M\mid 24$, there is a canonical  $\bar{\Q}$-isomorphism of vector spaces $$\mathfrak{S}_{r,s}\cong \mathfrak{H}^{2,0}(HD(r,s),1)_\Q.$$ 
\end{Lemma}
 From Lemma \ref{geom}, it is immediate that the dimension of $\mathfrak{S}_{r,s}$ is $\varphi(M)$. The Galois condition was used in \cite{EHMMI} to ensure that the associated hypergeometric Galois representations are extendable to $\mbox{Gal}(\ol\Q/\Q)$, which we denote by $G_\Q$ going forward. However, the definition from \cite{EHMMI} was ill-defined, and so we use the following from \cite{rosen2}.
\begin{definition}[\cite{rosen2}]\label{galoispre}
    A family $(r_j,s_j)$ of holomorphic conjugates is said to be \textit{Galois} if $r_j\neq r_i$ for all $j\neq i$.
\end{definition}
 In the following section, we explain that the Galois definition is better understood as a regularity condition on the underlying Galois representations.

\subsection{Regular and Irregular Families}\label{regiireg}
The purpose of this section is to relate the Hodge-theoretic discussion from the first half of this paper with the terminology used in the previous papers in the EHMM series, specifically \cite{EHMMI,rosen2}. The main result is Proposition \ref{gal-reg}, but note that this is not strictly necessary for the proofs later in the paper. We first recall the definition of regular motives. \begin{definition}[See e.g. \cite{patrikistaylot}]
    A hypergeometric datum is said to be \textit{regular} if the Hodge numbers of $\mathfrak{H}(c\cdot HD,t)$ are are all less than or equal 1, or equivalently, if $\rho_{c\cdot HD,\ell,t}$ has distinct Hodge--Tate weights, for all $c\in (\Z/M\Z)^\times$.
\end{definition}

Assume $f$ is a newform, and $\mathcal V_f$ denotes the $\bar{\Q}$-vector space generated by $f\otimes \chi$ for $\chi\in \widehat{(\Z/M\Z)^\times}$. We recall the following.
\begin{theorem}\label{cons}[\cite{rosen2}]
If $\mathfrak{S}_{r,s}$ is a \textit{Galois} family, then there exists a newform $f_{r,s}$ so that $$\mathcal{V}_{f_{r,s}}=\mathfrak{S}_{r,s}.$$ When $\varphi(M)>2$, then $f_{r,s}$ is non-CM.
\end{theorem}
 The essential point of this section is that Theorem \ref{cons} still holds if we replace the adjective Galois with regular. The set $\mathbb S_2^{irr}$ from Theorem \ref{lvalirr} is defined as $$\mathbb S_2^{irr}=\{(r,s)\in \mathbb S_2\,|\, \mathfrak{S}_{r,s} \text{ is irregular and primitive}\}.$$ Recall the definition of primitive is given in Definition \ref{ar}. We prove an analogue of Theorem \ref{cons} in Theorem \ref{cons2} below.

 In general, regular motives or Galois representations are much better understood; for example, there are powerful automorphy lifting theorems for hypergeometric Galois representations that only apply if the representation is regular, see \cite{patrikistaylot}. For this reason, much of the literature focuses on regular hypergeometric motives, which include the Dwork family studied in \cite{dwork-p_adiciccycles,katz-dwork-family,kedlaya-frobenius}. See also \cite{ahlgrenono,allen-supercongruences,EHMMI,asairec,EHMMIII,rodriguez-villegas-111,supercongr,patrikistaylot} for further examples. The irregular families are rarely if ever studied. In the forthcoming preprint \cite{LLL}, a few examples over $\Q$ are provided.  As noted in the previous section, the existing literature by the author and her collaborators on the EHMM have confusing and conflicting definitions of regularity. The purpose of the Galois families (Definition \ref{galoispre}) of \cite{EHMMI,rosen2} for the $\mathfrak{H}(HD(r,s),1)$ families was effectively a regularity condition. In this section, we briefly reformulate the Galois definition as a statement about regularity so that the other papers in the EHMM series fit more consistently in the literature.

To make sense of this, note that the regularity condition for $HD(r,s)$ is very simple, since the rank of the Hodge structure is small. At the parameter $t=1$, the Hodge number $h^{1,1}$ is reduced by one. 
\begin{Lemma}\label{galois}
   Let  $0<r<s\leq 1$. The hypergeometric Betti--de Rham structure $\mathfrak{H}(HD(r,s),t)$ has Hodge numbers $(1,1,1)$ if $s>1/2$, and Hodge numbers $(2,1,0)$ otherwise. If $s<r$, the Hodge numbers are $(1,1,1)$ if $s<1/2$ and $(0,1,2)$ otherwise. 
\end{Lemma}
\begin{proof}
We use the zigzag procedure \cite{fedorov,LLL,longyang}, which tells us to order the parameters by size, and then move up one unit for each parameter in $\alpha$ and move down one unit for each parameter of $\beta$. If $s>1/2$, since $r<s$ the ordering of the parameters will be either: $\{r,1/2,1/2,s,1,1\}$ or $\{1/2,1/2,r,s,1,1\}$, both of which have three $\alpha$ terms followed by three $\beta$ terms. Thus, the zigzag diagram will be of the shape \begin{center}
    \includegraphics[scale=.3]{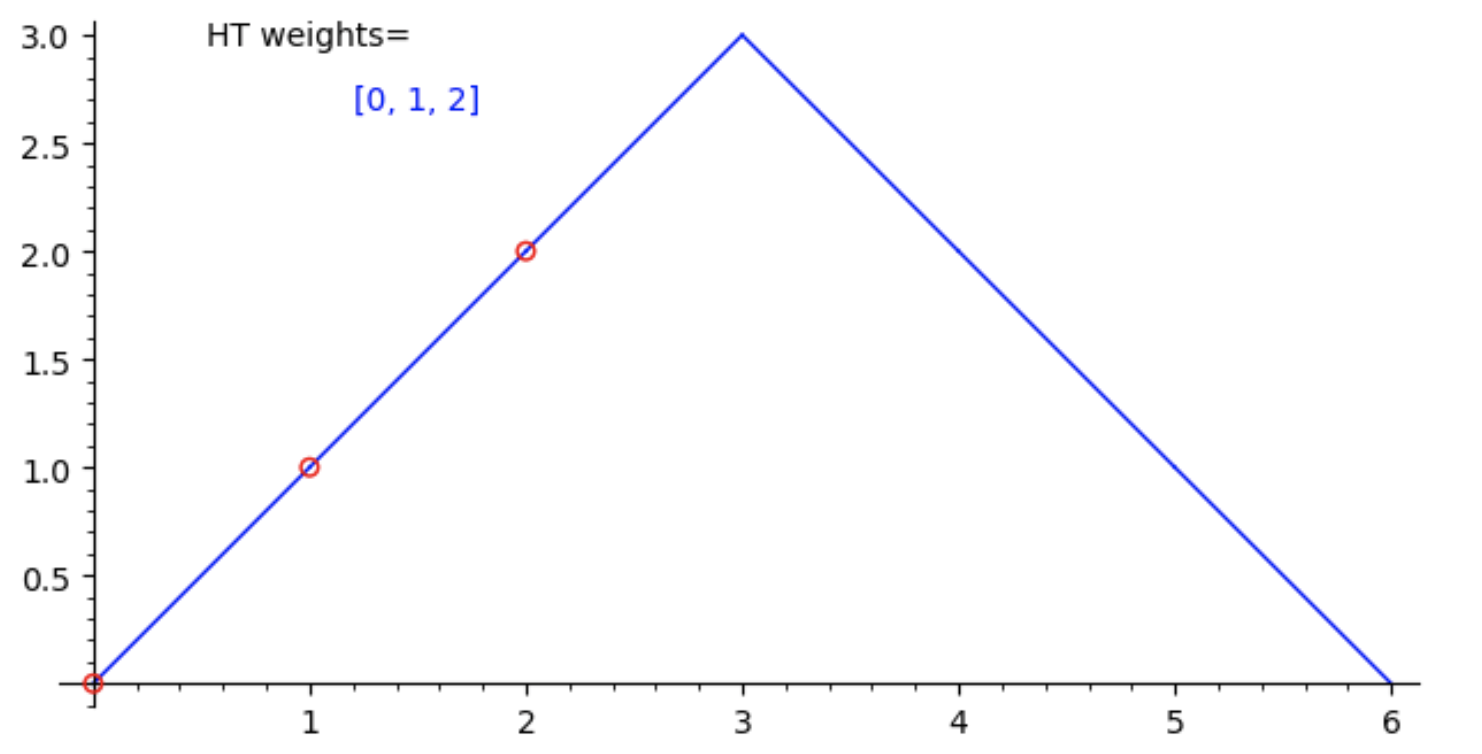}
\end{center}
with Hodge numbers $(1,1,1)$. Conversely, if $s<1/2$, we will necessarily have the ordering $\{r,s,1/2,1/2,1,1\}$, which means there is an $\alpha$ term followed by a $\beta$ term. Thus, the zigzag diagram has the form
\begin{center}
    \includegraphics[scale=.3]{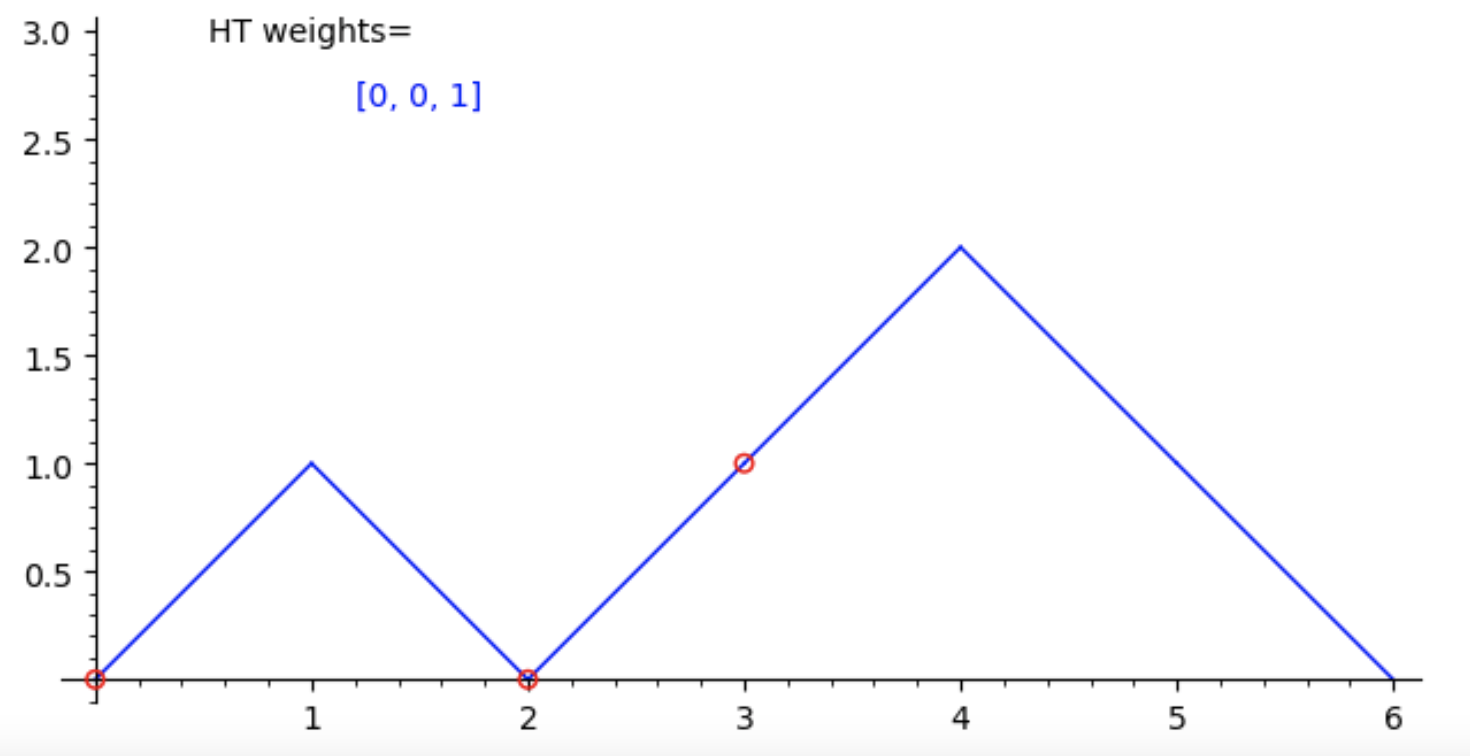}
\end{center}
with Hodge numbers $(2,1,0)$. For the second claim, note $\mathfrak{H}(HD(r,s_1),1)\oplus\mathfrak{H}(HD(1-r,1-s_1),1)$ is defined over a totally real field and so satisfies Hodge symmetry. The Hodge structure has effective weight 2 by work of Beukers--Heckman \cite{beukersheckman}, i.e. the Hodge numbers $h^{2,0}$ and $h^{0,2}$ cannot both be zero. Furthermore the rank of $\mathfrak{H}(HD(1-r,1-s_1),1)$ is 3, so the claim follows.
\end{proof}

As a result, the relationship with Definition \ref{galoispre} is as follows.

\begin{proposition}\label{gal-reg} Assume $\{(r_j,s_j)\}$ is the set of conjugate pairs corresponding $\mathfrak{H}^{2,0}(HD(r,s),1)$, and further assume that $\varphi(M)>2$. Let $b$ denote the denominator of $r_j$, which does not depend on $j$. Let $R_j=\{(r_i,s_i)\in \mathfrak{S}_{r,s}\,|\, r_i=j/b\}$. Then a family is regular if and only if $|R_i|=|R_j|$ for all $i,j\in (\Z/b\Z)^\times$. In particular, if the family is Galois $|R_j|=1$ for all $j$, and so every Galois family is regular.
\end{proposition}
\begin{proof}
Assume the family is regular.  Then the vector spaces $V_n^{2,0}$ are each one-dimensional, and the pair $(r_n,s_n)$ corresponds to an explicit differential $\omega_n\in V_n^{2,0}$. Thus, there is a well-defined transitive action of $(\Z/M\Z)^\times$ on the set of conjugates, $\{(r_j,s_j)\}_{j\in (\Z/M\Z)}$ given by multiplying $(r_j,s_j)$ by $n$ and taking the the unique holomorphic conjugate $(r_{k},s_k)$ that lies in $V_{jn}$. Clearly, $$|R_j|=|\{a\in (\Z/M\Z)^\times\,:\, ar_j\equiv r_j\mod \Z\}|.$$  By Lemma \ref{gal-reg}, $s_j>1/2$ for all conjugates, so we may assume that $r_j<1$ throughout, and since the denominator of $r_j$ is $b$, assume $r_j=x/b$, $x\in (\Z/b\Z)^\times$. So, the second set is precisely the stabilizer subgroup $G_x$ for the action of $(\Z/M\Z)^\times$ on its subgroup $(\Z/b\Z)^\times$. The orbits of this action are the cosets of $(\Z/b\Z)^\times$ in $(\Z/M\Z)^\times$, and thus by the Orbit--Stabilizer theorem, the stabilizer subgroups $G_x$ all have the same size.

Conversely, suppose the family is not regular. Then there exists some $(r,s_1)$ with $r=j/b$ so that $s_1<1/2$; therefore, the pair $(r,s_1+1)\in \mathbb S_2$, and it is also a conjugate with $n=x$, as the two are congruent mod $\Z$. Moreover, these pairs are linearly independent, and so must generate all of $\mathfrak{H}^{2,0}(HD(r,s),1)$, which has rank 2 by Lemma \ref{gal-reg}. Via this characterization, we deduce that if two pairs $(r_i,s_i)$ and $(r_k,s_k)$ both correspond to pairs in $\mathfrak{H}^{2,0}(HD(j/b,s_j),1)$, then $r_k=r_i$. Hence, each $R_j$ is in bijection with the direct sum of several $\mathfrak{H}^{2,0}(HD(j/b,s_j),1)$. Note that $\mathfrak{H}(HD(r,s_1),1)$ contributes 2 elements to $R_j$, but nothing to $R_{M-j}$  Any Hodge structures with Hodge numbers $(1,0,1)$ contributes to both $R_{M-j}$ and $R_j$ by the symmetry above. For the size of the two sets to be equal, $R_{M-j}$ must contain a Hodge structure $\mathfrak{H}(HD(1-r,s_2),1)$ with Hodge numbers $(2,0,0)$. Since $s_1<1/2$, we must have $r_1<1/2$, so $1-r>1/2$. However, by Lemma \ref{galois}, a family with $1-r>1/2$ can never have Hodge numbers $(2,0,0)$. Therefore, we conclude $|R_{M-j}|\neq |R_j|$, and the claim follows.

\end{proof}

\subsection{Constructing Modular Forms}\label{consmod}
The central approach of the EHMM is to compute a vector space of modular forms attached to the Hodge structure $\mathfrak{H}(HD(r,s),1)_\Q$. We complete this process in this subsection for the irregular families in $\mathbb S_2^{irr}$. In \cite{rosen2}, we addressed the case where the modular forms are non-CM, and in this paper, we focus on the CM cases. Recall $f$ has CM by $-D$ if $a_p(f)\left(\frac{-D}{p}\right)=a_p(f)$. Equivalently, we have $\rho_f\otimes \left(\frac{-D}{\cdot }\right)\cong \rho_f$ in terms of Galois representations. In particular, forces $a_p(f)=0$ for all $p$ that are not a square modulo $-D$. From Lemma \ref{coeffcong} and the fact that $M$ divides 24, this restricts the possible CM discriminants for our forms to $-3,-4,-8,-24$. We encounter forms will all five of these discriminants in this paper. Refer to Table \ref{tab} for the LMFDB labels of the forms constructed, up to a quadratic twist. 
\begin{table}[ht]
    \centering
    \begin{tabular}{c|ccc|ccccc}
        & Data  & $f_1$ & $f_2$ & Data & $f_1$ & $f_2$\\\hline
        1& $1/2,1$ & 16.3.c.a & - & - & - & -\\\hline
         2 & $i/4,3/2-i/4$ & 16.3.c.a & - & $i/4,1/2$ &64.3.c.a &-\\\hline
          3 & $i/3,3/2-i/3$ & 36.3.d.a &- & $i/3,2i/3$ &36.3.d.a &-\\\hline 
           4 & $i/8,i/8+1$ & 32.3.d.a & - & $i/8,1/2$ & 256.3.c.e & - \\\hline
           11 & $i/6,3/2-i/6$ & 144.3.g.c & - & - &- &-\\\hline
        12 & $i/24,i/24+[i]_6/6$ & 576.3.b.b & 576.3.b.c & $i/24,[i]_3/3$ &2304.3.g.t &2304.3.g.s\\\hline
        13 & $i/12,3/2-i/12$ &$\begin{array}{c} \text{36.3.d.a}\\\text{36.3.d.b}\end{array}$   & 144.3 g.c & $i/12,i/6$ & 576.3.g.c & 576.3.g.f \\\hline
        14 & $i/8,3/2-i/8$ & 128.3.d.a & 128.3.d.b & $i/8,i/4$ & 256.3.c.b & 256.3.c.d\\\hline
        15 & $i/24,3/2-i/24$ & 1152.3.b.b & 1152.3.b.f & $i/24,i/12$ &2304.3.g.e & 2304.3.g.i\\
    \end{tabular}
  
    \caption{CM $HD(r,s)$ families. The modular forms $f_1$ and $f_2$ are those constructed in Theorem \ref{cons2}}
    \label{tab}
\end{table}
We are now ready to prove the first part of Theorem \ref{lvalirr}. We restate a more specific version of the statement for the readers' convenience below. This version only applies to the primitive cases. The statement for the imprimitive case is given in Proposition \ref{consim}. We will take $V_n$ to denote the Betti--de Rham structure $\mathfrak{H}(HD([nr],[ns],1)$, and by an abuse of notation will identify $V_n^{2,0}$ with a space of modular forms via Lemma \ref{geom}. Recall also that $\mathcal{V}_f$ denotes the $\C$-vector space generated by $f$ and its twists by characters in $\widehat{(\Z/M\Z)^\times}$.

\begin{theorem}[Theorem \ref{lvalirr}, Part I]\label{cons2}
   Assume $(r,s)\in \mathbb S_2^{irr}$, and $\mathfrak{S}_{r,s}$ is generated by $\{\mathbb K_2(r_k,s_k)\}$, and let $M$ be the least common denominator of $r$, $s$, and $1/2$ as before. Then there exist two Hecke eigenforms with CM, $f_1$ and $f_2$ so that $$\mathfrak{S}_{r,s}=\mathcal{V}_{f_1}\oplus \mathcal{V}_{f_2}.$$ Explicitly, there are constants $\beta_{k,i}$ for $i=1,2$ so that $$f_1=\sum_{k=1}^{\varphi(M)}\beta_{k,1}\mathbb{K}_2(r_k,s_k)(N\tau)\quad\quad f_2=\sum_{j=1}^{\varphi(M)}\beta_{k,2}\mathbb{K}_2(r_k,s_k)(N\tau).$$ If $\varphi(M)\geq 4$, then $f_1$ is not a twist of $f_2$, and any other Hecke eigenform constructed as a linear combination of this family is a quadratic twist with conductor dividing $M$ of either $f_1$ or $f_2$. If $\varphi(M)=2$, then $f_1$ and $f_2$ are a quadratic twist of each other, and $f_1$ and $f_2$ are the unique two normalized eigenforms constructed from this family up to twisting by a character of conductor dividing $M$. 
\end{theorem}
\begin{remark} In light of Lemma \ref{geom}, this Theorem is equivalent to the first part of Theorem \ref{lvalirr}. Note also that $\beta_{k,i}$ will frequently be zero. However, there are situations where $\beta_{k,1}$ and $\beta_{k,2}$ are both nonzero.
\end{remark}
\begin{proof}
We can assume without loss of generality that if  $r=m/e$, $e=M$. This is because all pairs $(r,s)$ corresponding to a holomorphic and congruence $\mathfrak{S}_{r,s}$ families with $e<M$ satisfy $W_2\mathfrak{S}_{r,s}=\mathfrak{S}_{r',s'}$ for $r'=m/e$ and $M=e$. Since the Atkin--Lehner involution commutes with the Hecke operators, Theorem \ref{cons2} implies that the same two modular forms for $\mathfrak{S}_{r',s'}$ can be constructed for $\mathfrak{S}_{r,s}$. So henceforth, assume $M=e$, and write $$\mathfrak{S}_{r,s}\cong \bigoplus_{k\in (\Z/M\Z)^\times} V_k^{2,0},$$ where $V_k^{2,0}$ is identified with the space spanned by all $\mathbb{K}_2(r_k,s_k)(N\tau)$ so that $r_k=k/M$ via Lemma \ref{geom}. The $V_k^{2,0}$ are being viewed as a space of modular forms rather than differential 2-forms in this setting. By construction, $\dim V_k^{2,0}\leq 2$ for all $k$ since the Hodge structure $V_k$ over $\Q(\zeta_M)$ has rank 2. By Lemma \ref{galois}, $\dim V_k=2$ if and only if $\dim V_{M-k}=0$.  By an identical argument to Lemma 3.4 of \cite{rosen2} since $e=M$, if $p\equiv m\mod M$ each $V_k^{2,0}|T_p$ is equal to $V_{km}^{2,0}$, where $km$ is viewed as an element of $(\Z/M\Z)^\times$. Since $(\Z/M\Z)^\times$ is isomorphic to $\Z/2\Z$ raised to some power for all of our choices of $M$, this implies $V_k^{2,0}$ is stabilized by $T_p$ for $p\equiv 1\mod M$. Via the spectral theorem, $V_k^{2,0}$ has a basis of eigenforms for such $T_p$. Let us denote these two eigenforms by $g_{k,1}$ and $g_{k,2}$, with the assumptions $g_{k,1}=g_{k,2}$ if $\dim V_k=1$, $g_{1,k}=0$ if $\dim V_k^{2,0}=0$, and the first coefficient of all $g_{k,j}$ is normalized to one.

    We claim that if $g$ is any Hecke eigenform for all $T_p$ with $p\nmid M$ built from $\mathfrak{S}_{r,s}$, then $g$ must have CM.  We must have $\dim V_k=2$ for at least one $k$, and so $\dim V_{M-k}=0$ for at least one $k$ as well. By Definition \ref{conj}, if there is a $k$ so that $(k/M,s_k)$ has $s_k<1/2$, then the conjugate $(1/M,s_1)$ guaranteed by Lemma 3.1 of \cite{rosen2} satisfies $s_1>1$ or $s_1<1/2$. In either case, $\dim V_1=2$, and so we can assume without loss of generality that $k=1$. Due to Lemma \ref{coeffcong}, as $\dim V_{M-1}=0$, the Fourier coefficients $a_p(g)=0$ for all $p\equiv M-1\mod M$. By Dirichlet's theorem on primes in arithmetic progressions, this implies $a_p(g)=0$ only for a set of primes with positive density. On the other hand, the Sato--Tate conjecture for modular forms, now a theorem due to \cite{bght}, implies that if $f$ is a non-CM newform of weight $k\geq 2$, then $a_p(g)=0$ for a set of primes with density zero. Therefore, $g$ must have CM. Moreover, the prime $p$ can be inert in $\Q(\sqrt{-D})$ for all $p\equiv {M-1}\mod M$ only if $D\mid M$. Since the maximum value for $M$ is 24, the possible values for $D$ are $-3,-4,-8,-24$. On a case by case basis, we can check that there are exactly $\varphi(M)/2$ congruence classes mod $M$ of primes that are inert in $\Q(\sqrt{-D})$ for all $D\mid M$.
    
     By an appropriate combination of constants arising from Hecke eigenvalues for all $T_p$ as in Theorem 3.6 of \cite{rosen2}, $$f_j=\sum_{k\in (\Z/M\Z)^\times}\beta_{k,j}\cdot g_{j,k}$$ can be made a Hecke eigenform for all $p\nmid M$. Since $V_{1}^{2,0}$ is 2-dimensional with a basis of eigenforms for $p\equiv 1\mod M$, we may assume that $g_{1,1}\neq g_{1,2}$.  Normalize so that $\beta_{1,k}=1$. By the discussion above, we must have that $\beta_{k,j}=0$ for $\varphi(M)/2$ of the $k$. Therefore, $f_1$ and $f_2$ each have precisely $\varphi(M)/2$ quadratic twists with conductor dividing $M$, which modify the $\beta_{k,j}$ by a sign. As such, we have constructed $\varphi(M)$ normalized Hecke eigenforms given by $f_1$ and $f_2$ and their twists, and as the space $\mathfrak{S}_{r,s}$ has dimension $\varphi(M)$, this is a complete list.
     
     Finally, we need to determine if $f_1$ is a twist of $f_2$. For $\varphi(M)=2$, there is only one irregular and primitive family for $M=6$, and it can be checked that the only two newforms arising from it are a quadratic twist of each other. If $\varphi(M)\geq 4$, there is some $k\neq 1$ so that $\dim V_k^{2,0}=1$ by Lemma \ref{galois} with generator $g_{k,1}=g_{k,2}=g_k$. We know $\beta_{k,i}\neq 0$ for either $i=1$ or $i=2$; without loss of generality assume $\beta_{k,1}\neq 0$. Suppose by way of contradiction that $\beta_{k,2}\neq 0$ as well. Then the forms $g_{1,1}+\beta_{k,1}\cdot g_{k}$ and $g_{1,2}+\beta_{k,2}\cdot g_{k}$ are both fixed by the Hecke operators $T_p$ for $p\equiv 1\mod M$ by definition. Since, the Hecke operators are linear,
     \begin{align*}
         (g_{1,j}+\beta_{k}\cdot g_k)|T_p&=a_{p,j}\cdot  (g_{1,j}+\beta_{k,j}\cdot g_k)\\&=g_{1,j}|T_p+\beta_{k,j}\cdot g_k|T_p\\&=\tilde{b}_{p,j,1}\cdot g_{1,j}+\tilde{b}_{p,j,k}\cdot \beta_{k,j}\cdot g_k.
     \end{align*}
     So, $a_{p,i}=\tilde{b}_{p,j,1}=\tilde{b}_{p,j,k}$. However, note $\tilde{b}_{p,j,k}$ does not depend on $j$ because $g_k$ does not depend on $j$. This implies for all $p\equiv 1\mod M$, $a_{p,1}=a_{p,2}$. However, by Lemma \ref{coeffcong}, $g_{1,1}$ and $g_{1,2}$ have nonzero coefficients only if $p\equiv 1\mod M$ and since both are normalized so that the first coefficient is 1, $a_{p,j}=a_p(g_{1,j})$ for all $p\equiv 1\mod M$. Hence, $a_p(g_{1,1})=a_p(g_{1,2})$ for $p\equiv 1\mod M$, but both are zero for any other primes, and so it follows that $g_{1,1}=g_{2,1}$. This is a contradiction. Therefore, we must have that $\beta_{k,2}=0$. In other words, for $p\equiv k\mod M$, $f_1$ has nonzero coefficients and $f_2$ has zero coefficients. Therefore, $f_1$ and $f_2$ have different CM discriminants and so cannot be a twist of each other.

\end{proof}

\subsubsection{Example}\label{exee}
Theorem \ref{cons2} is far more technical than Theorem \ref{cons}. Essentially, when the family is regular each $V_k^{2,0}$ is a one-dimensional space, and so diagonalizing the action of the Hecke operators is far easier. We illustrate with an example. Consider the family containing the pair $(1/24,5/24)$. We find the following data: \begin{table}[ht]
    \centering
    \begin{tabular}{c|c|c|c|c}
       $n$ &1 &5 &7 &11\\\hline
       $\dim V_n^{2,0}$ &2 &1 &2& 1 
    \end{tabular}
    \caption{The dimensions of the spaces $V_n^{2,0}$ for $n\in (\Z/12\Z)^\times$. For the remaining $n\in (\Z/24\Z)^\times$, namely $13,17,19$, and $24$, the dimension can be computed via the symmetry $\dim V_n+\dim V_{24-n}=2$, appearing since each hypergeometric motive has rank 2.}
    \label{dimw}
\end{table}  Based on this, we can determine that $f_1$ and $f_2$ both have nonzero Fourier coefficients for $p\equiv 1\mod 24$ and $p\equiv 7\mod 24$. Note 7 splits in $\Q(i)$ and $\Q(\sqrt{-2})$ and is inert in $\Q(\sqrt{-3})$ and $\Q(\sqrt{-6})$. Since by Theorem \ref{cons2}, $f_1$ and $f_2$ have different CM discriminant, without loss of generality, we may assume $f_1$ has CM discriminant $-3$, and $f_2$ has CM discriminant $-24$. The other primes less than 24 which split in $\Q(\sqrt{-3})$ are $13$ and $19$, and so we can write $$f_1=g_{1,1}+\beta_{7,1}\cdot g_{7,1}+\beta_{13,1}\mathbb{K}_2(13/24,17/24)(24\tau)+\beta_{19,1}\mathbb{K}_2(19/24,23/24)(24\tau)$$ and likewise $$f_2=g_{1,2}+\beta_{5,2}\cdot \mathbb{K}_2(5/24,25/24)(24\tau)+\beta_{7,2}\cdot g_{7,2}+\beta_{11,2}\cdot \mathbb{K}_2(11/24,31/24)(24\tau).$$
For the $\beta_{k,j}$, we use the same process as in Theorem 3.6 of \cite{rosen2}, see also Section \ref{ex}. To complete the computation, we must determine $g_{7,i}$ and $g_{1,i}$. The process is identical for both. We illustrate with $g_{1,i}$. It suffices to compute the action of the Hecke operators $T_p$ where $p\equiv 1\mod{24}$ on the vector $$(K_1,K_2):=(\mathbb{K}_2(1/24,5/24)(N\tau),\mathbb{K}_2(1/24,29/24)(N\tau)).$$ This action is $$T_p\begin{pmatrix}
    K_1\\K_2
\end{pmatrix}=\begin{pmatrix}
    -4&5\\2&-1
\end{pmatrix}\begin{pmatrix}
    K_1\\K_2
\end{pmatrix}=\begin{pmatrix}
    g_{1,1}\\g_{2,1}
\end{pmatrix}.$$
 In total, this enables us to write down an explicit Fourier expansion for $f_1$ and $f_2$. We determine that the corresponding LMFDB labels are as listed in Table \ref{tab}.

\subsection{Evaluation Formulas and \texorpdfstring{$L$}{L}-Values}\label{leval}

We now show that there is a strong relationship between $L$-values of modular forms and evaluation formulas for classical hypergeometric series arising from irregular cases in $\mathbb S_2$. Throughout this section, we will assume $\{(r_j,s_j)\}$ is an irregular family of conjugates lying on $\mathbb S_2$, and $f_1$ and $f_2$ are the two Hecke eigenforms constructed in Theorem \ref{cons2}, with CM discriminant $-D_1$ and $-D_2$ respectively. We will use $i$ to denote the subscript $f_i$ for $i=1,2$ throughout this subsection. We assume for convenience that $e=M$, other cases can be handled using the Atkin--Lehner involution. To study $L$-values, we use the same argument as the previous section; however we reindex the sum so that it is over $j\in (\Z/M\Z)^\times$. 
Let $g_{i,k}$ be the linear combination of $\mathbb{K}_2(k/M,s_k)(N\tau)$ and $\mathbb{K}_2(k/M,s_k+1)(N\tau)$ associated to $f_i$ as in the proof of Theorem \ref{cons2}, and define the values \begin{equation}\label{fik}
    F_{i,k}=2\pi\int_0^{i\infty}g_{i,k}(\tau)d\tau=a_{i,1} F(k/M,s_k)+a_{i,2} F(k/M,s_k+1)
\end{equation} for some algebraic constants $a_{i,j}$.

\begin{Lemma}\label{chowlf}
    We have $F_{i,k}\in \pi^{-1}\Omega^2_{-D_i}\cdot \bar{\Q}$ for $F_{i,k}$ as in \eqref{fik}.
\end{Lemma}
\begin{proof}
   By construction, $$f_i=\sum_{k\in (\Z/M\Z)^\times } \beta_{i,j}\cdot g_{i,k}$$ for $\beta_{i,j}$ some algebraic integers.  From Theorem \ref{cons2}, we can also construct all $f_i\otimes \phi$ for $\phi$ a quadratic character of conductor dividing $M$ after multiplying $\beta_{i,j}$ by $\phi(j)$. Thus,  from \eqref{lval}, we have
   \begin{equation}\label{lvale}
       L(f_i\otimes \phi,1)=\sum_{k\in (\Z/M\Z)^\times} \phi(k)\cdot \beta_{i,k}\cdot F_{i,k}.
   \end{equation}  From CM theory \cite{damerell,grossror}, as $f_i$ is a CM modular form of discriminant say $-D_i$, we have $$L(f_i\otimes \phi,1)=\gamma_i(\phi)\cdot \pi^{-1}\Omega^2_{-D_1},$$ where $\gamma_i(\phi)\in \bar{\Q}$ depends on $i$ and $\phi$. Summing over all $\phi$ returns $$\sum_{\phi} L(f_i\otimes \phi,1)=\varphi(M)\beta_{i,1}F_{i,1}\in \pi^{-1}\Omega^2_{-D_i}\cdot \bar{\Q}$$ by above. On the other hand, if we only sum over half of the quadratic characters with conductor dividing $M$, this will produce $\varphi(M)\beta_{i,1}F_{i,1}/2+\varphi(M)\beta_{i,j}F_{i,j}/2$ for some $j\neq 1$. Since we already know $F_{i,1}\in \pi^{-1}\Omega^2_{-D_1}\cdot \bar{\Q}$ and the twisted $L$-values belong to the same ring, so $F_{i,j}\in \pi^{-1}\Omega^2_{-D_1}\cdot \bar{\Q}$ as well. Repeating the same argument enough times as necessary, we prove the Lemma for all $j$.
\end{proof}

The relation with $L$-values is obvious from the proof of Lemma \ref{chowlf}. We make the connection more explicit in the next theorem, which is the main theoretical result of this section.

\begin{proof}[Proof of Theorem \ref{lvalirr} Part II]
    The proof is almost immediate from Lemma \ref{chowlf}. By that Lemma, we have $F_{i,k}=a_{i,1}F(k/M,s_k)+a_{i,2}F(k/M,s_k+1)\in \pi^{-1}\Omega^2_{-D_i}$. Therefore, $$F_{1,k}-\frac{a_{1,1}}{a_{2,1}}F_{2,k}=\left(a_{1,2}+a_{2,2}\frac{a_{1,1}}{a_{2,1}}\right )F(k/M,s_k).$$ Since all of the $a_{i,j}$ are algebraic numbers, the first claim follows. The numbers $\beta_{i,j}$ are known from computing the action of the Hecke operators, so we ignore them. Now, by the proof in Lemma \ref{chowlf}, it is obvious that given the $L$-values for all $L(f_i\otimes \phi,1)$, we can determine all of the values $F_{i,j}$, and therefore $\alpha_{i,c}$ can be determined via the procedure above. Conversely, given the $\alpha_{i,c}$, we can reverse engineer the above in a straightforward way, which determines the $F_{i,j}$, and therefore the $L$-values from \eqref{lvale}.
\end{proof}

 For the example in Section \ref{exee}, the datum $(1/24,5/24)$, the two modular forms $f_1$ and $f_2$ have CM discriminants $-3$ and $-24$ respectively. Therefore, we have showed that $$F(1/24,5/24)=\alpha_{1}\cdot \Omega_{-3}+\alpha_2\cdot \Omega_{-24},$$ and similarly for $F(1/24,29/24)$. There are no known evaluation formulas for these data. Numerically, once can check that $F(1/24,29/24)/F(1/24,5/24)$ seems to be not algebraic, and also all of the $\alpha_i$ are nonzero for both cases. As such, it is difficult to guess the $\alpha_i$, even with knowledge of the theorem above. Determining these numerically would require computing the approximate $L$-value in \verb|Magma| for all of the twists and using the strategy outlined above. We leave such a computation and proving the correct values to future work. For now, we content ourselves with a few explicit examples in the enxt section.

\section{Examples}\label{ex}
We provide the examples promised in the introduction, which proves Theorem \ref{complex} for the irregular cases. The basic idea is to use Dixon's formula \eqref{dixon}. We then use a trick using $L$-values obtained from the EHMM introduced first in \cite{rosen2} to extract the $L$-values, as well as the explicit versions of our evaluation formulas as in \eqref{prelev}. We first note Dixon's formula is equivalent to the following.
\begin{corollary}\label{dixon2}
    Assume $r$ is a positive, non-integer, rational number. If $s=3/2-r$ or $s=2r$, we have $$F(r,s)=2^{1-6r}\cdot \alpha_r\cdot \pi^{-1}B(r,1/4)^2,$$ where $$\alpha_r=\begin{cases}
        \frac{\sin \pi r}{1+\sin 2\pi r} & \text{if } s=3/2-r\\
        1 & \text{if } s=2r.
    \end{cases}$$
\end{corollary}
\begin{proof}
We explain the case where $s=3/2-r$. The case $s=2r$ can be derived using \cite[Corollary 3.3.5]{aar} from the $s-3/2-r$ case, but we leave the details to the reader. We already have the evaluation \eqref{dixon}, and multiply by $B(r,3/2-2r)$ to rewrite in terms of $F(r,s)$. The denominator of $B(r,3/2-2r)=\frac{\Gamma(r)\Gamma(3/2-2r)}{\Gamma(3/2-r)}$ cancels the $\Gamma(3/2-r)$ term in the numerator of \eqref{dixon}. We can use the reflection formula and functional equation to group the combine each of $\Gamma(5/4)$ and $\Gamma(3/4)$, as well as $\Gamma(1-r)$ and $\Gamma(r)$, so that a $\Gamma(r)^2\Gamma(1/4)^2$ appears in the numerator of \eqref{dixon}. After applying the functional equation, the remaining from in \eqref{dixon} is $\Gamma(3/4-r)\Gamma(3/2-2r)/\Gamma(5/4-r)$. Using the duplication formula, we can write $$\Gamma(3/2-2r)=2^{-1/2+2r}\Gamma(3/4-r)\sqrt{\pi}\Gamma(5/4-r).$$ The $\Gamma(5/4-r)$ cancel, and we may move $\Gamma(3/4-r)^2$ to the denominator using the reflection formula, becoming $\Gamma(1/4+r)^2$. By \eqref{betagamma}, we now have \eqref{dixon} is equal to $B(1/4,r)^2$ up to an algebraic multiple and a power of $\pi$, assuming $r\not\in \Z$. Carefully keeping track constants and powers of $\pi$, we obtain the corollary.
\end{proof}

\begin{remark}
    The similarity between this proof and Lemma \ref{deco} should be apparent. This illustrates the general philosophy that a proof about gamma quotients can be converted into a proof about Gauss sum quotients and vice versa. A motivic version of such results is explained by Otsubo--Yamazaki \cite{otsubo-yamazaki-motivic-gauss} in special cases.
\end{remark}

\subsection{Example 1}
The one irregular case where $\varphi(M)=2$ behaves differently, simply because there are not many conjugates. There are three pairs $(r,s)\in \mathbb S_2$ conjugate to this family, but one is a linear combination of the other two. 

The three holomorphic conjugates satisfy the relation
\begin{equation*}
      \mathbb{K}_2(1/6,1/3) =  \mathbb{K}_2(1/6,4/3)+ 16\mathbb{K}_2(7/6,4/3),
\end{equation*}
 which is easily proved using Fourier expansions. As an identity of hypergeometric functions, this is a contiguous relation, which are well-studied (see e.g. \cite{contiguous}). \begin{equation}\label{contiguous}
    -F(1/6,1/3)+F(1/6,4/3)+ 16F(7/6,4/3)=0.
\end{equation} 
Both $F(1/6,4/3)$ and $F(1/6,1/3)$ can be evaluated using Lemma \ref{dixon2}, and so the contiguous relation enables us to evaluate $F(7/6,4/3)$. Below is the explicit version of Theorem \ref{lvalirr}. 

\begin{Lemma}\label{class5}
    We have $$f_{144.3.g.c}=\frac{\mathbb{K}_2(1/6,4/3)(12\tau)+\mathbb{K}_2(1/6,1/3)(12\tau)}{2}-8\sqrt{-3}\cdot\mathbb{K}_2(7/6,4/3)(12\tau).$$ Moreover, $$L(f_{144.3.g.c},1)=\frac{(-1)^{1/4}}{12\cdot 2^{2/3}\cdot  3^{3/4}} \pi^{-1} \Omega_{-3}^2.$$
\end{Lemma}
\begin{proof}
   The first equality follows from Theorem \ref{cons2} after computing the action of the Hecke operators. For the $L$-value we can use \eqref{lval} and the first part to write the $L$-value as an explicit linear combination of hypergeometric series. From there, we must write $F(1/6,1/3)$ and $F(1/6,4/3)$ as an algebraic multiple of $\pi^{-1}\Omega_{-3}^2$. Then we may use the contiguous relation \eqref{contiguous} to reach the result. By Corollary \ref{dixon2}, each of these is equal to $\pi^{-1}B(1/4,1/6)^2$ up to an explicit algebraic number. 
   
 We now write $B(1/4,1/6)^2$ as an algebraic number times $\Omega_{-3}$. This calculation is quite technical, so we postpone the proof to Lemma \ref{betaval} in the Appendix.
\end{proof}

\subsection{Example 2}\label{ex2}

   For a second example consider $(r,s)\in \mathbb{S}_2$ conjugate to $(1/8,3/8)$. The other irregular families all behave somewhat like this example. The subset of $\mathbb S_2^{irr}$ mentioned in Theorem \ref{lvalirr} is $$\{(1/8,11/8),\,(5/8,7/8),\,(3/8,9/8),\,(1/8,3/8),\,(9/8,11/8)\}\subset \mathbb S^{irr}_2.$$ The next Lemma follows from Corollary \ref{dixon} and \cite{proc,rosen2}.
\begin{Lemma}\label{ngal}
   $F(1/8,11/8)$ and $F(5/8,7/8)$ are in $\Bar{\Q}\cdot \pi^{-1}\Omega_{-4}^2$ and $F(3/8,9/8)\in \Bar{\Q}\cdot \pi^{-1} \Omega_{-8}^2$. Moreover, $$F(1/8,3/8)=F(1/8,11/8)+16F(9/8,11/8).$$
\end{Lemma}
\begin{proof}
    The hypergeometric series $F(i/8,3/2-i/8)$ can be evaluated as a beta value using Corollary \ref{dixon2}, namely $B_i=B(1/4,i/8)$. It is proved in \cite{proc} that $B_1$ and $B_5$ are an algebraic multiple of $\Omega_{-4}$, and $B_3$ is a multiple of $\Omega_{-8}$; refer to Table \ref{tab:placeholder3} for convenience. The three-term identity is a contiguous relation. There are classical proofs of such identities \cite{contiguous}, but a simpler proof using $\mathbb{K}_2(r,s)$ functions can be used in this case, as in \cite{rosen2}. The idea is to prove the identity $\mathbb{K}_2(1/8,3/8)=\mathbb{K}_2(1/8,11/8)+16\mathbb{K}_2(9/8,11/8)$ using Fourier coefficients and then integrate both sides.
\end{proof}
For example, we can check that $$2\sqrt{-1+\sqrt{2}}F(1/8,11/8)=\pi^{-1} \Omega_{-4}^2.$$ The two Chowla--Selberg periods appearing in the Lemma correspond to the two modular forms 128.3.d.a and 128.3.d.b listed in the table for non-Hecke stable classes, which have CM discriminant $-4$ and $-8$ respectively.

The following explicit calculation is a corollary of Dixon's formula, see \cite{aar}, combined with Theorem \ref{lvalirr}.
\begin{corollary}\label{local}
$$L(f_{128.3.d.a},1)=F(1/8,11/8)+8iF(5/8,7/8)=\frac{1}{4} \sqrt{-\frac{1}{2} + \frac{i}{2}}\pi^{-1}\Omega_{-4}^2.$$   
\end{corollary}
\begin{proof}
   We check $$f_{128.3.d.a}=\mathbb{K}_2(1/8,11/8)(8\tau)+8i\cdot\mathbb{K}_2(5/8,7/8)(8\tau)$$ and $$f_{128.3.d.b}=\mathbb{K}_2(1/8,3/8)(16\tau)+16\mathbb{K}_2(9/8,11/8)(16\tau)+4\sqrt{2}\cdot\mathbb{K}_2(3/8,9/8)(16\tau)$$ by computing the action of the Hecke operators. Write $F(1/8,11/8)$ and $F(5/8,7/8)$ as an algebraic multiple of $\pi^{-1}\Omega_{-4}^2$ using Lemma \ref{ngal}, then apply \eqref{lval} and simplify.
\end{proof}

We can similarly compute from \eqref{lval} that $$L(f_{128.3.d.b},1)=F(1/8,3/8)+16\cdot F(9/8,11/8)+4\sqrt{2}\cdot F(3/8,9/8).$$ There is an evaluation for $F(3/8,9/8)$ in Lemma \ref{ngal}, but not for the other two. To circumvent this, we can use a three-term identity, a classical type of symmetry going back to Thomae \cite{thomae}. The strategy we use to prove the necessary identity is similar to \cite{rosen2}.

\begin{Lemma}\label{threeterm2}
    $F(1/8,3/8)+16\cdot F(9/8,11/8)=4(2+\sqrt{2})\cdot F(3/8,9/8)$
\end{Lemma}
\begin{proof}
    There is a unique non-trivial inner twist of $f_{128.3.d.b}=f$, called $\phi$, which switches the sign of $\sqrt{2}$. By a theorem of Shimura \cite{shimura}, there exists an $\alpha'=g(\phi)\cdot \alpha$ so that \begin{align*}
        L(f,1&)=F(1/8,3/8)+16\cdot F(9/8,11/8)+4\sqrt{2}\cdot F(3/8,9/8)\\&=\alpha'\cdot [F(1/8,3/8)+16\cdot F(9/8,11/8)-4\sqrt{2}\cdot F(3/8,9/8)]=\alpha'\cdot L(f\otimes \phi,1).
    \end{align*}
    Rearranging, we arrive at the identity $$F(1/8,3/8)+16\cdot F(9/8,11/8)=\frac{4\sqrt{2}(1+\alpha')}{1-\alpha'}F(3/8,9/8).$$ As $F(3/8,9/8)\in \pi^{-1}\Omega_{-8}^2\cdot \bar{\Q}$, we already know the linear combination has the desired transcendental part. Furthermore, we can determine the exact value of $\alpha'$ using the second part of Shimura's theorem. Via the exact same argument as Lemma 4.7 in \cite{rosen2}, we determine that $\alpha'=-1+\sqrt{2}$. Unraveling, we get the identity.
\end{proof}

This provides us the information we need to compute the $L$-value.

\begin{corollary}\label{44}
    $$L(f_{128.3.d.b},1)=\frac{2\sqrt{2(1+\sqrt{2})}}{18}\cdot \pi^{-1}\Omega_{-8}^2.$$
\end{corollary}
\begin{proof}
    We know $$L(f_{128.3.d.b},1)=F(1/8,3/8)+16\cdot F(9/8,11/8)+4\sqrt{2}\cdot F(3/8,9/8)=4(2+2\sqrt{2})F(3/8,9/8)$$ from Lemma \ref{threeterm2}. The result then follows by Lemma \ref{ngal}.
\end{proof}

We also provide the proof of \eqref{prelev} \begin{corollary}\label{eval}
    We have $$\pfq{3}{2}{1/2,1/2,1/8}{1,1,3/8}{1}=\frac{1}{2\pi}(\Omega_{-8}^2/\Omega_{-4}+(-1+\sqrt{2})\cdot\Omega_{-4}),$$ or equivalently, $$F(1/8,3/8)=\pi^{-1}\sqrt{1+\sqrt{2}}\cdot\Omega_{-8}^2+\pi^{-1}\sqrt{-1+\sqrt{2}}\cdot\Omega_{-4}^2.$$
\end{corollary} 

\begin{proof}
    By the Corollary \ref{44}, Corollary \ref{local}, and the argument of Lemma \ref{threeterm2}, we know the $L$-values for $f_1$ and $f_2$, as well as all of their twists. Thus, by Theorem \ref{lvalirr}, the value $F(1/8,3/8)$ is known. We compute this explicitly as $$4F_{1,1}=L(f_1,1)+L(f_1\otimes \phi,1)\quad\quad 4F_{2,1}=L(f_2,1)+L(f_2\otimes \psi,1),$$ where $\phi$ and $\psi$ are the unique nontrivial inner twists acting on $f_1$ and $f_2$ as in the proof of Lemma \ref{threeterm2}. From the proof of Corollary \ref{local}, $f_{2,1}=\mathbb{K}_2(1/8,11/8)(16\tau)$ and $f_{1,1}=\mathbb{K}_2(1/8,3/8)(16\tau)+16\mathbb{K}_2(9/8,11/8)(16\tau)$. The linear dependence $${\mathbb{K}_2(1/8,3/8)(16\tau)-\mathbb{K}_2(1/8,11/8)(16\tau)}=16\mathbb{K}_2(9/8,11/8)(16\tau),$$ corresponding to a contiguous relation as in Lemma \ref{contiguous}, means we can rewrite $f_{1,1}$ as $$f_{1,1}=2\mathbb{K}_2(1/8,3/8)(16\tau)-\mathbb{K}_2(1/8,11/8)(16\tau).$$ As a result, $$f_{1,1}+f_{2,1}=2\mathbb{K}_2(1/8,3/8)(16\tau),$$ and integrating either side lets us conclude $$F_{1,1}+F_{2,2}=2F(1/8,3/8),$$ or $$F(1/8,3/8)=\frac{1}{8}[L(f_1,1)+L(f_1\otimes \phi,1)+L(f_2,1)+L(f_2\otimes \psi,1)].$$ Combined with Corollary \ref{44} and Corollary \ref{local}, we determine the values of $\alpha_{1,1}$ and $\alpha_{2,1}$ are as claimed in the statement.
\end{proof}

A similar evaluation can be found for $F(1/12,5/12)$ via an identical argument. We omit the details. The family of conjugates associated to $(1/24,11/24)$ has $\dim V_k^{2,0}=2$ for some $k\neq 1$, and so further work would be necessary to make the constants in Theorem \ref{lvalirr} explicit. For the non-well-poised family $(1/24,5/24)$, a completely new method would be required.

\begin{appendix}
\section{\texorpdfstring{$L$}{L}-values for Imprimitive Examples}\label{imprim}
From the examples in this paper, the $L$-values for $\mathbb K_2(r,s)(N\tau)$ functions are essentially completely classified, in combination with \cite{EHMMII,rosen2}. The only cases not treated in either of those papers or the main text of this paper are those where the hypergeometric datum is imprimitive. For completeness, we illustrate the $L$-values for these families. The imprimitive hypergeometric data are an example of a generically reducible family of hypergeometric motives. In particular, the motives should decompose into the direct sum of hypergeometric motives with shorter length datum and specific rank 1 motives. The datum $HD(r,s)$ for $(r,s)\in \mathbb S_2$ is imprimitive if and only if $s=r+1$, $r=1$, or $s=1/2$. It is immediate from the definition that $${}_3F_2(HD(1,s),1)=\pfq{2}{1}{1/2,1/2}{1,s}{1}\quad\quad {}_3F_2(HD(r,1/2),1)=\pfq{2}{1}{1/2,r}{1,1}{1}.$$ For $(r,r+1)$, Otsubo \cite{asakuraotsubo-cmperiods} provides an evaluation up to an algebraic multiple. However, we circumvent the need for Otsubo's work using the Atkin--Lehner involution in our cases. We observe in this section that in the framework of the EHMM, the well-poised cases behave like a weight 3 analogue of the $\mathbb K_1$ function associated to ${}_2F_1(1)$ introduced in \cite{proc}, which correspond to Fermat motives. The piece of these motives corresponding to the ${}_2F_1(1)$ are a first version of the CM motives that should appear in the decomposition of primitive reducible hypergeometric motives. As a byproduct, we also obtain the $L$-values for several additional weight 3 modular forms.

Recall from \cite{proc} that the space of conjugate $\mathbb K_1$ functions has dimension $\varphi(M)/2$, rather than $\varphi(M)$ as in the primitive weight 3 case. Analogously, if $HD(r,s)$ is imprimitive, then $$\dim \mathfrak{S}_{r,s}=\varphi(M)/2.$$ As such, any modular forms constructed from these families will be CM. There is the following analogue of Theorem \ref{cons} and Theorem \ref{cons2}.

\begin{proposition}\label{consim}
    Assume that $HD(r,s)$ is imprimitive. Then there exists a weight 3 modular form $f_{r,s}$ with CM discriminant $-M$ so that $$\mathfrak{S}_{r,s}=\mathcal{V}_{f_{r,s}}.$$
\end{proposition}
\begin{proof}
    The existence statement is immediate from the fact that $\dim \mathfrak{S}_{r,s}=\varphi(M)/2,$ and the fact that $\mathfrak{S}_{r,s}$ is stabilized by the Hecke operators from \cite{rosen2}. The CM discriminant can be checked by Lemma \ref{coeffcong} and noting that holomorphic conjugates can occur only when $r<1/2.$
\end{proof}

\begin{proposition}
    The exact $L$-value for all $f_{r,s}$ constructed in Proposition \ref{consim} can be written as an explicit multiple of the Chowla--Selberg period $\Omega_{-M}$, where $M$ as always is the least common denominator of $r$ and $s$. In particular, if $\sigma$ denotes a quadratic twist of conductor 8 switching the sign of 3, we have $$L(f_{1/8,1/2},1)=\frac{(-1)^{3/8}\Omega_{-8}^2}{4\pi}\quad\quad L(f_{1/8,1/2}^\sigma,1)=\frac{-(-1)^{5/8}\Omega_{-8}^2}{4\pi} $$ 

    $$L(f_{1/4,1/2},1)=\frac{1}{4\sqrt{2}\pi}\Omega_{-4}^2\quad\quad  L(f_{1/8,9/8},1)= \frac{1}{8\sqrt{2}\pi}\Omega_{-8}^2$$
\end{proposition}

\begin{remark}
    Note that the modular form $f_{1/8,9/8}$ has no nontrivial quadratic twists fixing the level, and so we only obtain one $L$-value for this form.
\end{remark}

\begin{proof}
    The basic idea is the same as \cite{rosen2}. For example, we can write $f_{1/4,1/2}=\mathbb{K}_2(1/4,1/2)(8\tau)$, and so the $L$-value is found by evaluating $$F(1/4,1/2)=\frac{1}{8}B(1/4,1/4)\pfq{2}{1}{1/2,1/4}{1,1}{1}=\frac{1}{8}\frac{B(1/4,1/4)^2}{B(1/4,3/4)}$$ using the cancellation and Gauss evaluation \eqref{gausseval}. In the Appendix, we see that $B(1/4,1/4)=\sqrt{2}\Omega_{-4}$, and by the reflection formula, $B(1/4,3/4)=\sqrt{2}\pi$, and so the claim for this $L$-value follows. The argument for the two $(1/8,1/2)$ cases is similar, except we must first write $f_{1/8,1/2}$ as the linear combination $\mathbb{K}_2(1/8,1/2)(16\tau)\pm 2i\mathbb{K}_2(3/8,1/2)(16\tau)$. Apply a similar analysis to above and the results follow. Finally, for $(1/8,9/8)$, there is not an obvious evaluation formula, but we can circumvent this issue by observing that $\mathfrak{S}_{1/8,9/8}|W_2=\mathfrak{S}_{1,9/8}$, and so $f_{1/8,1/2}=f_{1,9/8}$. We compute $f_{1,9/8}=\mathbb{K}_2(1,9/8)(2\tau)/2+\mathbb{K}_2(1,11/8)(2\tau)/2$. From here, we need to compute $$F(1,9/8)=\frac{1}{16}B(1/8,9/8)\pfq{2}{1}{1/2,1/2}{1,9/8}{1}=B(1/8,9/8)\frac{B(1/8,3/8)}{B(1,1/8)}.$$ By the functional equation, $2B(1/8,9/8)=B(1/8,1/8)$. From Table \ref{tab:placeholder3}, both $B(1/8,1/8)$ and $B(1/8,3/8)$ are equal to $\Omega_{-8}$ up to an algebraic number, which is ca easily be computed by hand. Repeating this proves for $F(1,11/8)$ completes the proof.
\end{proof}

    \section{Converting Beta Values to Chowla--Selberg Periods}\label{tables}

The general rule in the study of periods, especially  $L$-values, is that the transcendental part is predictable following the conjectures of Deligne \cite{deligne-l-values}, especially in the CM cases, where we have the Gross--Deligne conjecture which is proved in many cases; refer to \cite{grossror} for a survey. By contrast, the algebraic part encodes mysterious arithmetic information, which for elliptic curves is the content of the generalized Birch and Swinnerton--Dyer conjecture \cite{bsd}, and more generally the Bloch--Kato conjecture. Even the simplest cases of these conjectures remain unknown. On a smaller scale, we encounter the same problem when examining periods involving values of the beta and gamma function. In particular, the transcendental part of the period can be proven by theoretical means, such as \cite{grossror} in our cases. For example, Theorem \ref{lvalirr} tells us that the values $F_{i,j}$ are a multiple of some Chowla--Selberg period, finding the algebraic part is a difficult task. Furthermore, given a beta value or product of beta values, determining how to write this as a multiple of the expected transcendental part can be an irritating process, especially if the expected transcendental part is not immediately obvious. Even when the transcendental part is obvious, actually proving the value of the algebraic part can require some ingenuity. We provide a demonstration below.

\begin{Lemma}\label{betaval}
    We have that $$B(1/4,1/6)^2=2^{1/3}\sqrt{-\frac{27}{4} + \frac{9\sqrt{3}}{2}}\cdot \Omega_{-3}^2.$$
\end{Lemma}
One could prove the quotient $B(1/4,1/6)/\Omega_{-3}$ is algebraic using the Deligne--Koblitz--Ogus relations \cite[Appendix]{deligne-l-values}, but this would not provide the exact algebraic relation which is crucial for our computations of $L$-values.
\begin{proof}
    We use the three relations in Section \ref{gamma}. We observe that $$B(1/2,1/6)^2/\Omega_{-3}^2=\frac{\sqrt{3}\Gamma(1/6)^2\Gamma(2/3)^2}{2\pi \Gamma(1/3)^2}=2^{1/3}\sqrt{3}$$ after an application of the duplication formula. Finally, we note that $$B(1/2,1/6)^2/B(1/4,1/6)^2=\frac{\pi \Gamma(5/12)^2}{\Gamma(1/4)^2 \Gamma(2/3)^2}.$$ For the final part, we use multiplication by 3 to get that $$\Gamma(1/12)\Gamma(5/12)\Gamma(3/4)=2\pi\cdot  3^{1/4}\Gamma(1/4)$$ and replace one of the $\Gamma(1/4)$ with this so that the other $\Gamma(1/4)$ will cancel with $\Gamma(3/4)$ via the reflection formula. Canceling and applying the reflection and duplication formula repeatedly, compute 
  \begin{align*}
      B(1/2,1/6)^2/B(1/4,1/6)^2&=\frac{2\cdot  3^{1/4}\pi^2 \Gamma(5/12)^2}{\Gamma(1/4)\Gamma(1/12)\Gamma(5/12)\Gamma(3/4) \Gamma(2/3)^2}\\&
      =\frac{2\cdot 3^{1/4}\pi\Gamma(5/12)}{\sqrt{2}\Gamma(1/12)\Gamma(2/3)^2}=\frac{2\cdot 3^{1/4}(-1 + \sqrt{3})\Gamma(5/12)\Gamma(11/12)}{4\Gamma(2/3)^2}\\&
      =2\cdot 3^{1/4}(-1 + \sqrt{3})\cdot 2^{1/6} \sqrt{\pi} \frac{\Gamma(5/6)}{\Gamma(2/3)^2}\\&
      =2\cdot 3^{1/4}(-1 + \sqrt{3})\cdot 2^{1/6} \cdot\sqrt{3}/2\sqrt{\pi} \frac{\Gamma(5/6)\Gamma(1/3)}{\pi\Gamma(2/3)}\\&
      =2\cdot 3^{1/4}(-1 + \sqrt{3})\cdot 2^{1/6}\cdot\sqrt{3}/2\cdot 2^{1/3}=\sqrt{-\frac{9}{4} + \frac{3\sqrt{3}}{2}}.
  \end{align*}
\end{proof}

For convenience of exposition in this paper, we list beta values of the form $B(i/24,j/24)$ which are multiples of Chowla--Selberg periods in Table \ref{tab:placeholder3}. We do not claim this is an exhaustive list. The transcendental parts listed in the table are provable, either via \cite{proc}, Lemma \ref{betaval}, or ``trivial'' computations. We consider the functional equation, the reflection formula, and the duplication formula to be trivial because, as illustrated in the final computation of Lemma \ref{betaval}, we can immediately tell when they can be used in all cases, and so can be implemented algorithmically for example in Mathematica or Maple.

Some of these computations are easy; for instance,
\begin{align*}
     B(1/8,3/8)=2^{3/4} \Omega_{-8}\\
    B(1/6,1/6)=2^{5/6}3^{1/4}\Omega_{-3}\\
     B(1/4,1/4)=\sqrt{2}\Omega_{-4}
\end{align*}
are easy to show using the trivial properties, or at worst multiplication by 3.

\begin{table}[H]
    \centering
    \begin{tabular}{|c|ccc|}
        \hline Period & $(r,s)$ & Beta value & Equivalents \\\hline
        $\Omega_{-24}$ & $(11/12,35/24)$ & $B(1/24,11/12)$ & $2^{5/6}B(1/12,11/24)$  \\
       & $(11/12,35/24)$ & $B(1/24,11/12)$ &$2^{-1/12}B(1/24,11/24)$ \\
         & $(5/12,29/24)$ & $B(7/24,5/12)$ & $2^{1/6}B(5/24,7/12)$ \\
          & $(5/12,29/24)$ & $B(7/24,5/12)$ & $2^{-7/12}B(7/24,5/24)$ \\
         & $(1/24,5/6)$ & $B(1/24,7/24)$ & $2^{1/3}(2+\sqrt{2})B(5/24,11/24)$  \\
         & $(1/24,7/12)$ & $B(1/24,1/24)$ & $2^{5/6}\cot(\pi/24)B(11/24,11/24)$  \\
         & $(5/24,11/12)$ & $B(5/24,5/24)$ & $2^{1/6}\cot(5\pi/24)B(7/24,7/24)$  \\
         & $(1/24,29/24)$ & $B(1/24,2/3)$ & $2^{1/3}(1+\sqrt{3})B(5/24,1/3)$  \\
         & $(1/24,29/24)$ & $B(1/24,2/3)$ & $2^{1/3}(2+\sqrt{6})B(11/24,2/3)$ \\
         & $(1/24,29/24)$ & $B(1/24,2/3)$ & $\cos(5\pi/24)\csc(\pi/24)B(17/24,1/3)$\\\hline
         $\Omega_{-8}$& $(1/8,1)$  & $B(1/8,3/8)$ & $2^{3/4}\sin(\pi/8)B(1/8,1/8)$  \\
         & $(1/8,1)$  & $B(1/8,3/8)$ & $2^{-1/4}\csc(\pi/8)B(3/8,3/8)$ \\
         & $(1/8,1)$ & $B(1/8,3/8)$ & $\sec(\pi/8)B(1/8,1/2)$  \\
        & $(1/8,1)$ & $B(1/8,3/8)$ &  $\csc(\pi/8)B(3/8,1/2)$ \\
        & $(1/8,1)$ & $B(1/8,3/8)$ &  $2^{3/4}B(3/8,1/4)$ \\
         & $(1/8,1)$ & $B(1/8,3/8)$ & $2^{1/4}B(1/8,3/4)$  \\\hline
         $\Omega_{-3}$& $(1/6,5/6)$ & $B(1/6,1/6)$ & $2^{1/3}\sqrt{3}B(1/3,1/3)$  \\
          & $(1/6,5/6)$ & $B(1/6,1/6)$  & $2\sqrt{3}B(1/6,1/3)$  \\
         & $(1/6,5/6)$ & $B(1/6,1/6)$ & $2^{-1/6}B(1/12,7/12)$  \\
         & $(1/24,17/24)$ & $B(1/24,1/6)$ & $\csc(\pi/4)\sin(\pi/4)B(19/24,1/6)$ \\
         & $(1/24,17/24)$ & $B(1/24,1/6)$ &$2\sin(5\pi/4)B(1/24,19/24)$\\
         & $(7/24,23/24)$ &$B(7/24,1/6)$& $\cos(\pi/4)\sec(5\pi/4)B(13/24,1/6)$ \\
         & $(7/24,23/24)$ &$B(7/24,1/6)$& $2\cos(\pi/24) B(7/24,13/24)$\\\hline
         $\Omega_{-4}$ & $(1/4,1)$ & $B(1/4,1/4)$ & $\sqrt{2}B(1/4,1/2)$ \\
          & $(1/4,1)$ & $B(1/4,1/4)$ & $2^{3/4}\sin(\pi/8) B(1/8,1/4)$\\
         & $(1/4,1)$ & $B(1/4,1/4)$ & $\sqrt{1+\sqrt{2}} B(5/8,1/4)$\\
          & $(1/4,1)$ & $B(1/4,1/4)$ & $2^{3/4}\sin(\pi/8) B(1/8,1/4)$\\
         & $(1/4,1)$ & $B(1/4,1/4)$ & $2^{-1/4} B(1/8,5/8)$\\
         & $(1/12,2/3)$ & $B(1/12,1/12)$ & $2^{2/3}(2+\sqrt{3}) B(5/12,5/12)$ \\
         & $(1/24,23/24)$& $B(1/24,5/12)$ & $2^{1/3}\cos(\pi/4)\sec(5\pi/4) B(5/24,1/12)$ \\
         & $(1/24,23/24)$& $B(1/24,5/12)$  & $\cot(\pi/4)B(13/24,5/12)$\\
         & $(1/24,23/24)$& $B(1/24,5/12)$ & $2^{1/3}\cos(\pi/4)\csc(5\pi/4)B(17/24,1/12)$  \\ 
         & $(1/24,23/24)$& $B(1/24,5/12)$ &  $2\sqrt{2}\cos(\pi/24)(1+\sqrt{3})^{-1}B(1/24,13/24)$  \\          
         &$(1/24,23/24)$&$B(1/24,5/12)$&$2^{-2/3}\cos(\pi/24)B(5/24,17/24)$ \\\hline
    \end{tabular}
    \caption{Beta values that are algebraic multiples of a Chowla--Selberg period. The equivalents are algebraic multiples of the beta value listed under Beta value obtained \textit{trivially}, that is using only the functional equation, reflection formula, and duplication formula. Note that by the theory of CM abelian varieties, if $B(r,s)$ is a multiple of $\Omega_{-D}$, then $B(1-r,1-s)$ is a multiple of $\pi/\Omega_{-D}$.}
    \label{tab:placeholder3}
\end{table}

 However the remaining cases in Table \ref{tab:placeholder3} are substantially harder. As shown in Lemma \ref{betaval},  such proofs require ingenuity that Mathematica and Maple seem unable to determine on their own. It would be interesting to try and find a more effective algorithm that would enable more efficient computation for such problems. 

In this appendix, we only address relations with Chowla--Selberg periods. However, it in principle the same ideas should work for arbitrary Gross--Deligne periods. For example, consider the beta value $B(1/5,3/5)=\frac{\Gamma(1/5)\Gamma(3/5)}{\Gamma(4/5)}$. This should be a period for some field $K$ contained in $\Q(\zeta_5)$. The only subfield of $\Q(\zeta_5)$ is $\Q(\sqrt{5})$, and so the period cannot be any Chowla--Selberg period, which are associated to imaginary quadratic fields. As such, this should be equal to one of the Gross--Deligne periods. In practice, the main obstacle for computing this is that the function $\varepsilon(a/\sigma)$ from Conjecture \ref{grossdeligne-l-values} is not very explicit. We leave calculations involving Gross--Deligne periods to future work. Compare also to \cite{asakuraotsubo-cmperiods}.

\end{appendix}
\bibliographystyle{plain} 
\bibliography{mixed_cm}{}

\end{document}